\documentclass[11pt,reqno,twoside]{amsart}

\usepackage{amsfonts,bm,commath,mathrsfs,mathtools}
\usepackage{enumerate,enumitem,esint}
\usepackage{graphicx}
\usepackage{hyperref}
\usepackage{comment}
\usepackage{cite}

\usepackage[dvips,bottom=1.4in,right=1in,top=1in, left=1in]{geometry}

\DeclareMathAlphabet{\mathpzc}{OT1}{pzc}{m}{it}

\usepackage[textsize=small]{todonotes}
\numberwithin{equation}{section}

\theoremstyle{plain}
\newtheorem{corollary}{Corollary}

\newtheorem{keyLemma}{Key-Lemma}

\newtheorem{mainThrm}{Main Theorem}

\theoremstyle{definition}
\newtheorem{definition}{Definition}
\newtheorem{remark}{Remark}

\newtheorem{crem}{Concluding Remark}

\definecolor{wineRed}{rgb}{0.7,0,0.3}
\definecolor{grandBleu}{rgb}{0,0,0.8}
\definecolor{darkGreen}{rgb}{0,0.4,0}
\definecolor{blueViolet}{rgb}{0.4,0,1.0}
\definecolor{bloodOrange}{rgb}{0.85,0.05,0}

\def\laplace{\mathit{\Delta}}
\def\bbN{{\mathbb{N}}}
\def\bbR{{\mathbb{R}}}

\def\scrH{{\mathscr{H}}}
\def\scrV{{\mathscr{V}}}

\def\laplace{\mathit{\Delta}}
\def\ds{\displaystyle}

\title[A CLASS OF PARABOLIC SYSTEMS IN OPTIMAL CONTROL OF GRAIN BOUNDARY MOTIONS]{A CLASS OF PARABOLIC SYSTEMS ASSOCIATED \\ WITH OPTIMAL CONTROLS \\ OF GRAIN BOUNDARY MOTIONS}

\author{Harbir Antil}
\address{Department of Mathematical Sciences, George Mason University, Fairfax, VA 22030, USA.}
\email{hantil@gmu.edu}

\author{Ken Shirakawa}
\address{Department of Mathematics, Faculty of Education, Chiba University 1--33 Yayoi-cho, Inage-ku, 263--8522, Chiba, Japan}
\email{sirakawa@faculty.chiba-u.jp}

\author{Noriaki Yamazaki}
\address{Department of Mathematics, Faculty of Engineering, Kanagawa University 3--27--1 Rokkakubashi, Kanagawa-ku,Yokohama, 221--8686, Japan}
\email{sirakawa@faculty.chiba-u.jp}

\thanks{This work is supported by  NSF grants DMS-1818772, DMS-1521590, and Grant-in-Aid for Scientific Research (C) No. 16K05224, JSPS}



\begin{document}
\label{page:t}

\begin{abstract}
We propose a semi-discrete numerical scheme and establish well-posedness of a class of parabolic systems. Such systems naturally arise while studying the optimal control of grain boundary motions. The latter is typically described using a set of parabolic variational inequalities. We use a regularization approach to deal with the variational inequality. The resulting optimization problem is a nonsmooth, nonconvex, and nonlinear programming problem. This is a long term project where in the current work we are first analyzing systems of PDEs associated with the regularized optimal control problem. Such a system is a set of highly coupled parabolic equations, and proposes significant analytical and numerical challenges. We establish well-posedness of this system. In addition, we design a provably convergent semi-discrete (time discrete spatially continuous) numerical scheme to solve the system. We have developed several new tools during the course of this paper that can be applied to a wider class of coupled systems.
\end{abstract}

\maketitle

\section*{Introduction}\label{Intro}

Let $ (0, T) \subset \bbR$ be a bounded time-interval with a finite constant $ T > 0 $. Let  $ \Omega \subset \bbR^N $ be a bounded spatial domain with a dimension $ N \in \{1, 2, 3\} $. We denote by $ \Gamma $ the boundary $ \partial \Omega $ of $ \Omega $, and when $ N > 1 $, we suppose the Lipschitz regularity for $ \Gamma $. We denote by $ n_\Gamma \in \mathbb{S}^{N -1} $ the unit outer normal on $ \Gamma $, and set $ Q := (0, T) \times \Omega $ and $ \Sigma := (0, T) \times \Gamma $. Besides, we define $ H := L^2(\Omega) $ and  $ \scrH := L^2(0, T; L^2(\Omega)) $  as the base spaces for our problems. 
\medskip

In this paper, we consider a class of systems of parabolic initial-boundary value problems. Each constituent system is denoted by (\hyperlink{P}{P}), and formulated as follows. 
\begin{align*}
    \mbox{(\hypertarget{P}{P}):} &  
\nonumber
\\
& \parbox{12cm}{
$ \begin{cases}
    \partial_t p -\laplace p +\mu(t, x) p +\lambda(t, x) p +\omega(t, x) \cdot \nabla z 
\\
\hspace{1.75ex} = h(t, x), \mbox{ $ (t, x) \in  Q $,}
\\
\nabla p \cdot n_\Gamma = 0 \mbox{ on $ \Sigma $,}
\\
p(0, x) = p_0(x), ~ x \in \Omega;
\end{cases} $
}
\end{align*}
\vspace{-6ex}
\begin{align*}
\hspace{4ex} &  
\nonumber
\\
& \parbox{12cm}{
$ \begin{cases}
    \ds a(t, x) \partial_t z + b(t, x) z -\mathrm{div} \bigl( A(t, x) \nabla z +\nu \nabla z +p \omega(t, x) \bigr)
\\
\hspace{1.75ex}  = k(t, x), \mbox{ $ (t, x) \in Q $,}
\\
z = 0 \mbox{ on $ \Sigma $,}
\\
z(0, x) = z_0(x), ~ x \in \Omega.
\end{cases} $
}
\end{align*}
The unknown of the system (\hyperlink{P}{P}) is a pair of functions $ [p, z] : Q \longrightarrow \bbR^2 $ with $ [p_0, z_0] : \Omega \longrightarrow \bbR^2 $ denoting the initial data, and $ [h, k] : Q \longrightarrow \bbR^2 $ denoting the forcing terms. Moreover, $ \nu > 0 $ is a fixed constant, and $ a : Q \longrightarrow (0, \infty) $, $ b : Q \longrightarrow \bbR $, $ \mu : Q \longrightarrow [0, \infty) $, $ \lambda : Q \longrightarrow \bbR $, $ \omega : Q \longrightarrow \bbR^N $, and $ A : Q \longrightarrow \bbR^{N \times N} $ are given functions. 
\medskip

System (\hyperlink{P}{P}) is associated with an optimal control problem for the following  ``regularized" state system, denoted by (\hyperlink{S}{S}).
\begin{align}\label{S_1st}
    \mbox{(\hypertarget{S}{S}):} &  
    \nonumber
    \\
    & \parbox{12cm}{
    $ \begin{cases}
        \partial_t \eta -\laplace \eta +g(\eta) +\alpha'(\eta)f(\nabla \theta) = u(t, x), \mbox{ $ (t, x) \in  Q $,}
    \\
        \nabla \eta \cdot n_\Gamma = 0 \mbox{ on $ \Sigma $,}
        \\
        \eta(0, x) = \eta_0(x), ~ x \in \Omega;
    \end{cases} $
    }
\end{align}
\vspace{-6ex}
\begin{align}\label{S_2nd}
    \hspace{4ex} &  
    \nonumber
    \\
    & \parbox{12cm}{
    $ \begin{cases}
        \ds \alpha_0(t, x) \partial_t \theta -\mathrm{div} \left( \alpha(\eta) \nabla f(\nabla \theta) +\nu \nabla \theta \right) = v(t, x), \mbox{ $ (t, x) \in Q $,}
        \\
        \theta = 0 \mbox{ on $ \Sigma $,}
        \\
        \theta(0, x) = \theta_0(x), ~ x \in \Omega.
    \end{cases} $
}
\end{align}
This state system is a regularized version of a mathematical model of grain boundary motion, called ``Kobayashi--Warren--Carter model'', which was proposed by Kobayashi et al. \cite{MR1752970,MR1794359}. In their original works \cite{MR1752970,MR1794359}, the spatial domain $ \Omega $ is taken to be a two-dimensional domain, and the dynamics of grain boundaries are represented by a vector field
$$
(t, x) \in Q \mapsto \eta(t, x) \left[ \rule{0pt}{10pt} \cos \theta(t, x), \sin \theta(t, x) \right],
$$
consisting of two unknowns $ \eta : Q \longrightarrow \bbR $ and $ \theta : Q \longrightarrow \bbR $. The unknowns $ \eta $ and $ \theta $, with the initial data $ \eta_0 : \Omega \longrightarrow \bbR $ and  $ \theta_0 : \Omega \longrightarrow \bbR $, denote the order parameters which indicate the orientation order and the orientation angle of the grain, respectively. The functions $ u : Q \longrightarrow \bbR $ and $ v: Q \longrightarrow \bbR $ are given forcing terms. The Lipschitz function $ g : \bbR \longrightarrow \bbR $ in \eqref{S_1st} helps us to control the range of $ \eta $. Moreover, $ \alpha_0 : \bbR \longrightarrow (0, \infty) $ and  $ \alpha : \bbR \longrightarrow (0, \infty) $ in \eqref{S_2nd} are given functions that correspond to the mobilities of grain boundaries, and $ \alpha' $ in \eqref{S_1st} denotes the differential of $ \alpha  $. Finally, $ f : \bbR^N \longrightarrow [0, \infty) $ is a given $ C^2 $-convex function, and $ \nabla f $ denotes the gradient. 
\medskip

Notice that the $C^2$-function $f$ in \eqref{S_1st} and \eqref{S_2nd} is not present in the original 
works \cite{MR1752970,MR1794359}, instead the authors use a $ N $-dimensional Euclidean norm. This is the main reason we call \eqref{S_1st} and \eqref{S_2nd} a ``regularized" state system, and in this regard, we call the corresponding optimal control problem as regularized optimal control problem. In case of the Euclidean norm, the diffusion profile is given in a singular form $ -\mathrm{div} \bigl( \alpha(\eta) \frac{\nabla \theta}{|\nabla \theta|} +\nu \nabla \theta \bigr) $, where such a singularity is desirable to reproduce the facet-like situations as in the practical crystalline structures. However, the new diffusion profile after introducing $f$ is given by $ -\mathrm{div} \bigl( \alpha(\eta) \nabla f(\nabla \theta) +\nu \nabla \theta \bigr) $. Thus in order to capture the facet-like situations using the regularized model, it is crucial to ensure that $f$ is a good approximation of the Euclidean norm. We shall elaborate on this further in a forthcoming paper. 
\medskip

For completeness, we state the regularized optimal control problem next. We shall denote it 
by (\hyperlink{OCP}{OCP}).
\begin{description}
    \item[\textmd{\textrm{(\hypertarget{OCP}{OCP})}}]Find a pair of functions $ [u_*, v_*] \in [\scrH]^2 $, called an \emph{optimal control}, such that 
\begin{equation*}
[u_*, v_*] \in \mathscr{U}_\mathrm{ad} := \left\{ \begin{array}{l|l}
[u, v] \in [\scrH]^2 & 
\begin{minipage}{5.5cm}
$ u_a \leq u \leq u_b $ and $ v_a \leq v \leq v_b $, a.e. in $ Q $
\end{minipage} 
\end{array} \right\}(\ne \emptyset),
\end{equation*}
and 
$ [u_*, v_*] $ minimizes the following \emph{cost functional} $ \mathcal{J} = \mathcal{J}(u, v)  $:
    \begin{align*}
        [u, v] \in & \mathscr{U}_\mathrm{ad} \subset [\scrH]^2 \mapsto \mathcal{J}(u, v) 
        \nonumber
        \\
        := & \ds \frac{1}{2} \int_0^T \bigl( |(\eta -\eta_\mathrm{ad})(t)|_{H}^2 +|(\theta -\theta_\mathrm{ad})(t)|_{H}^2 \bigr) \, dt 
        \\
        & \ds +\frac{\lambda_1}{2} \int_0^T |u(t)|_{H}^2 \, dt  +\frac{\lambda_2}{2} \int_0^T |v(t)|_{H}^2 \, dt\in [0, \infty),
        \nonumber
    \end{align*}
where $ [\eta, \theta] $ solves the regularized state system (\hyperlink{S}{S}). 
\end{description}
Here, $ [u_a, v_a], [u_b, v_b] \in L^\infty(Q)^2 $  are fixed bounded obstacles for the variable $ [u, v] \in [\scrH]^2 $, called a \emph{control}, and are known as \emph{control bounds}, $ \lambda_i > 0 $, $ i = 1, 2 $, are constants of control cost, and $ [\eta_\mathrm{ad}, \theta_\mathrm{ad}] \in [\scrH]^2 $ is the pair of \emph{admissible target profiles} for $ [\eta, \theta] $. 
\medskip

The problem (\hyperlink{P}{P}) generalizes the following two key-systems that arise in the mathematical analysis of (\hyperlink{OCP}{OCP}):
\begin{description}
    \item[\hypertarget{sharp1}{$\sharp\,1)$}]the linearized system for (\hyperlink{S}{S}) which is associated with the G\^{a}teaux derivative $ \mathcal{J}' $ of the cost $ \mathcal{J} $ at each $ [u, v] \in \mathscr{U}_\mathrm{ad} $;
    \item[\hypertarget{sharp2}{$\sharp\,2)$}]the adjoint system of the above problem $ \hyperlink{sharp1}{\sharp\,1}) $.
\end{description}
In the context, problem $ \hyperlink{sharp1}{\sharp\,1}) $ corresponds to the system (\hyperlink{P}{P}) in the case when:
\begin{equation}\label{relat_P-S01}
    \begin{cases}
        a(t, x) = \alpha_0(t, x), ~ b(t, x) = 0,
        \\[1ex]
        \mu(t, x) = \bigl[ \alpha''(\eta) f(\nabla \theta) \bigr](t, x),
        \\[1ex]
        \lambda(t, x) = [g'(\eta)](t, x), 
        \\[1ex]
        \omega(t, x) = \bigl[ \alpha'(\eta) \nabla f(\nabla \theta) \bigr](t, x),
        \\[1ex]
        A(t, x) = \bigl[ \alpha(\eta) \nabla^2 f(\nabla \theta) \bigr](t, x),
        \\[1ex]
        [p_0(x), z_0(x)] = [0, 0],
    \end{cases}
    \mbox{a.e. $ (t, x) \in Q $,}
\end{equation}
and the problem $ \hyperlink{sharp2}{\sharp\,2}) $ corresponds to the system (\hyperlink{P}{P}) in the case when:
\begin{equation}\label{relat_P-S02}
    \begin{cases}
        a(t, x) = \alpha_0(T -t, x), ~ b(t, x) = \partial_t \alpha_0(T -t, x),
        \\[1ex]
        \mu(t, x) = \bigl[ \alpha''(\eta) f(\nabla \theta) \bigr](T -t, x),
        \\[1ex]
        \lambda(t, x) = [g'(\eta)](T -t, x), 
        \\[1ex]
        \omega(t, x) = \bigl[ \alpha'(\eta) \nabla f(\nabla \theta) \bigr](T -t, x),
        \\[1ex]
        A(t, x) = \bigl[ \alpha(\eta) \nabla^2 f(\nabla \theta) \bigr](T -t, x), 
        \\[1ex]
        [p_0(x), z_0(x)] = [0, 0],
    \end{cases}
    \mbox{a.e. $ (t, x) \in Q $,}
\end{equation}
where $ g' $ is the differential of the Lipschitz perturbation $ g $. 
\medskip

The goals of this paper are to establish well-posedness of the system (\hyperlink{P}{P}), and to design a convergent numerical scheme (time-discrete, spatially continuous) to numerically realize the problem. We shall accomplish our goals in the following two theorems.

\begin{description}
    \item[\textbf{\textrm{Main Theorem \hyperlink{mainThrm01}{1} \textmd{(Well-posedness of (\hyperlink{P}{P}))}:}}]to show existence, uniqueness, and continuous dependence on the data for (\hyperlink{P}{P}).
\item[\textbf{\textrm{Main Theorem \ref{mainThrm2} \textmd{(Numerical scheme and convergence)}:}}]to set up a time-discretization for (\hyperlink{P}{P}) and establish convergence of this scheme.
\end{description}

The content of this paper is as follows. The Main Theorems of this paper are stated in Section 2, on the basis of the preliminaries in Section 1. The proofs of Main Theorems are given in Section 4, by means of the Key-Lemmas provided in Section 3. Finally, in Section 5, we show several Corollaries and Remarks as the conclusions derived from the Main Theorems and Key-Lemmas.

\section{Preliminaries}\label{01Prelim}

We begin by prescribing the notations used throughout this paper. 
\medskip

As is mentioned in the previous Section, let $ (0, T) \subset \bbR$ be a bounded time-interval with a finite constant $ T > 0 $. Let  $ \Omega \subset \bbR^N $ be a bounded spatial domain with a dimension $ N \in \{1, 2, 3\} $. We denote by $ \Gamma $ the boundary $ \partial \Omega $ of $ \Omega $, and when $ N > 1 $, we suppose the Lipschitz regularity for $ \Gamma $. We denote by $ n_\Gamma \in \mathbb{S}^{N -1} $ the unit outer normal on $ \Gamma $, and set $ Q := (0, T) \times \Omega $ and $ \Sigma := (0, T) \times \Gamma $. Besides, we define  
\begin{equation*} \begin{cases}
H := L^2(\Omega) \mbox{ and } \scrH := L^2(0, T; H),
\\
V := H^1(\Omega) \mbox{ and } \scrV := L^2(0, T; V),
\\
V_0 := H_0^1(\Omega) \mbox{ and } \scrV_0 := L^2(0, T; V_0),
\end{cases}
\end{equation*}
and we suppose that:
\begin{equation*}
\begin{cases}
V \subset H = H^* \subset V^* \mbox{ and } \scrV \subset \scrH = \scrH^* \subset \scrV^*,
\\
V_0 \subset H = H^* \subset V_0^* \mbox{ and } \scrV_0 \subset \scrH = \scrH^* \subset \scrV_0^*,
\end{cases}
\end{equation*}
where ``$ = $'' is due to identification of the Hilbert spaces with their duals, and moreover, ``$ \subset $'' denotes continuous embeddings. 
\bigskip

\noindent
\underline{\textbf{\textit{Notations in real analysis.}}}
Let $ d \in \bbN $ take a fixed dimension. We denote by $ |x| $ and $ x \cdot y $ the Euclidean norm of $ x \in \mathbb{R}^d $ and the scalar product of $ x, y \in \bbR^d $, respectively, i.e., 
\begin{equation*}
\begin{array}{c}
| x | := \sqrt{x_1^2 +\cdots +x_d^2} \mbox{ \ and \ } x \cdot y  := x_1 y_1 +\cdots +x_d y_d, 
\\[1ex]
\mbox{ for all $ x = [x_1, \ldots, x_d], ~ y = [y_1, \ldots, y_d] \in \mathbb{R}^d $.}
\end{array}
\end{equation*}

We denote by $\mathcal{L}^{d}$ the $ d $-dimensional Lebesgue measure, and we denote by $ \mathcal{H}^{d} $ the $ d $-dimensional Hausdorff measure.  In particular, the measure theoretical phrases, such as ``a.e.'', ``$dt$'', and ``$dx$'', and so on, are all with respect to the Lebesgue measure in each corresponding dimension. Also in case of a Lipschitz-surface $ S $, the phrase ``a.e.'' is with respect to the Hausdorff measure in each corresponding Hausdorff dimension. 

Additionally, for a Borel set $ E \subset \bbR^d $, we denote by $ \chi_E : \bbR^d \longrightarrow \{0, 1\} $ the characteristic function of $ E $.
\bigskip

\noindent
\underline{\textbf{\textit{Abstract notations. (cf. \cite[Chapter II]{MR0348562})}}}
For an abstract Banach space $ X $, we denote by $ |\cdot|_{X} $ the norm of $ X $, and denote by $ \langle \cdot, \cdot \rangle_X $ the duality pairing between $ X $ and its dual $ X^* $. In particular, when $ X $ is a Hilbert space, we denote by $ (\cdot,\cdot)_{X} $ the inner product of $ X $. 

For two Banach spaces $ X $ and $ Y $,  let $  \mathscr{L}(X; Y)$ be the Banach space of bounded linear operators from $ X $ into $ Y $. 

For Banach spaces $ X_1, \dots, X_d $ with $ 1 < d \in \bbN $, let $ X_1 \times \dots \times X_d $ be the product Banach space endowed with the norm $ |\cdot|_{X_1 \times \cdots \times X_d} := |\cdot|_{X_1} + \cdots +|\cdot|_{X_d} $. However, when all $ X_1, \dots, X_d $ are Hilbert spaces, $ X_1 \times \dots \times X_d $ denotes the product Hilbert space endowed with the inner product $ (\cdot, \cdot)_{X_1 \times \cdots \times X_d} := (\cdot, \cdot)_{X_1} + \cdots +(\cdot, \cdot)_{X_d} $ and the norm $ |\cdot|_{X_1 \times \cdots \times X_d} := \bigl( |\cdot|_{X_1}^2 + \cdots +|\cdot|_{X_d}^2 \bigr)^{\frac{1}{2}} $. In particular, when all $ X_1, \dots,  X_d $ coincide with a Banach space $ Y $, we let:
\begin{equation*}
    [Y]^d := \overbrace{Y \times \cdots \times Y}^{\mbox{d times}},
\end{equation*}
and moreover, when there is no possibility of confusion, we simply denote by $ Y^d $ the product space of this case. 

For any proper functional $ \Psi : X \rightarrow(-\infty, \infty] $ on a Banach space $ X $, we denote by $ D(\Psi) $ the effective domain of $ \Psi $. 
\bigskip

\noindent
\underline{\textbf{\textit{Notations of basic differential operators.}}}
Let $ F: V \longrightarrow V^* $ be the duality map, defined as
\begin{equation}\label{dualMapF}
\langle F \varphi, \psi \rangle_{V} := (\nabla \varphi, \nabla \psi)_{[H]^N} + (\varphi, \psi)_{H}, \mbox{ for all $ \varphi, \psi \in V $.}
\end{equation}
As is well-known, 
\begin{equation*}
F \varphi = - \laplace \varphi +\varphi \mbox{ in $ H $, if $ \varphi \in H^2(\Omega) $, and $ \nabla \varphi \cdot n_\Gamma = 0 $ a.e. on $ \Gamma $.}
\end{equation*}

In the meantime, we note that the distributional divergence ``$ \mathrm{div} $'' in $ \mathscr{D}'(\Omega) $ can be regarded as a bounded linear operator from $ [H]^N $ into $ V_0^* $ (i.e. $ \mathrm{div} \in \mathscr{L}([H]^N; V_0^*) $), via the following variational identity: 
\begin{equation}\label{div_V0*}
\bigl\langle \mathrm{div}\, \varpi, \psi \bigr\rangle_{V_0} := -\bigl( \varpi, \nabla \psi \bigr)_{[H]^N}, \mbox{ for every $ \varpi \in [H]^N $ and $ \psi \in V_0 $.}
\end{equation}
\bigskip

\noindent
\underline{\textbf{\textit{Notations for the time-discretization.}}} 
Let $ \tau \in (0, 1) $ be a constant that denotes the time step-size, and let  $ \{ t_i \}_{i = 0}^{\infty} \subset [0, \infty) $ be the time sequence defined as:
\begin{equation} \label{seqTime}
t_i := i \tau, ~ i = 0, 1, 2, \dots.
\end{equation}
Let $ X $ be a Banach space. Then, for any sequence $ \{ [t_i, \gamma_i] \}_{i = 0}^\infty \subset [0, \infty) \times X $, we define the \emph{forward time-interpolation} $ [\overline{\gamma}]_\tau \in L_\mathrm{loc}^\infty([0, \infty); X)$, the \emph{backward time-interpolation} $ [\underline{\gamma}]_\tau \in L_\mathrm{loc}^\infty([0, \infty); X) $, and the \emph{linear time-interpolation} $ [{\gamma}]_\tau \in W_\mathrm{loc}^{1, 2}([0, \infty); X) $, by letting:
\begin{equation}\label{timeInterp}
\begin{cases}
~ \ds [\overline{\gamma}]_{\tau}(t) := \chi_{(-\infty, 0]}(t) \gamma_0 +\sum_{i = 1}^\infty \chi_{(t_{i -1}, t_i]}(t)\gamma_i,
\\
~ \ds [\underline{\gamma}]_{\tau}(t) := \sum_{i = 0}^\infty \chi_{(t_{i}, t_{i +1}]}(t) \gamma_i,
\\
    ~ \ds [{\gamma}]_{\tau}(t) := \sum_{i = 1}^\infty \chi_{[t_{i -1}, t_i)}(t) \left( \frac{t -t_{i -1}}{\tau} \gamma_i +\frac{t_i -t}{\tau} \gamma_{i -1} \right), 
\end{cases} \mbox{in $ X $, for $ t \geq 0 $,}
\end{equation}
respectively.

\begin{remark}\label{Rem.t-discrete}
    For an interval $ I \subset \bbR $,  a Banach space $ X $, and a constant $ q \in [1, \infty] $, we say that $ L^q(I; X) \subset L_\mathrm{loc}^q(\bbR; X) $ (resp. $ L_\mathrm{loc}^q(\bbR; X) \subset L^q(I; X) $) by identifying $ X $-valued functions on $ I $ (resp. on $ \bbR $) with the zero-extensions onto $ \bbR $ (resp. the restriction onto $ I $). Besides, under the notations as in \eqref{seqTime} and \eqref{timeInterp}, the following facts can be verified.
\begin{description}
    \item[{\textrm{\hypertarget{Fact1}{(Fact\,1)} $ \bullet $}}]If $ q \in [1, \infty) $, $ \gamma \in L^q(0, T; X) $, and the sequence $ \{ \gamma_i \}_{i = 0}^\infty \subset X $ is given by:
\begin{equation}\label{rApx}
\gamma_i := \frac{1}{\tau} \int_{t_{i -1}}^{t_{i}} \gamma(\varsigma) \, d \varsigma \mbox{ in $ X $,} ~ i = 0, 1, 2, \dots,
\end{equation}
then
\begin{equation*}
    [\overline{\gamma}]_\tau \to \gamma, ~ [\underline{\gamma}]_\tau \to \gamma, \mbox{ and } [\gamma]_\tau \to \gamma \mbox{ in  $ L^q(0, T; X) $, as $ \tau \downarrow 0 $,}
\end{equation*}
and
\begin{equation*}    
    \begin{array}{c}
        [\overline{\gamma}]_\tau(t) \to \gamma(t), ~ [\underline{\gamma}]_\tau(t) \to \gamma(t), \mbox{ and }[\gamma]_\tau(t) \to \gamma(t) 
        \\[1ex]
        \mbox{ in $ X $, a.e. $ t \in (0, T) $, as $ \tau \downarrow 0 $.}
    \end{array}
\end{equation*}
\item[~~~~$ \bullet $]If $ X $ is a reflexive Banach space, and $ \gamma \in L^\infty(0, T; X) $, then the sequence $ \{ \gamma_i \}_{i = 0}^\infty \subset X $ given by \eqref{rApx} fulfills that: 
\begin{equation*}
    \sup_{\tau \in (0, 1)} \bigl\{ |[\overline{\gamma}]_\tau|_{L^\infty(0, T; X)}, |[\underline{\gamma}]_\tau|_{L^\infty(0, T; X)}, |[\gamma]_\tau|_{L^\infty(0, T; X)} \bigr\} \leq |\gamma|_{L^\infty(0, T; X)},
\end{equation*}
\begin{equation*}
    \begin{cases}
        [\overline{\gamma}]_\tau \to \gamma, ~ [\underline{\gamma}]_\tau \to \gamma, \mbox{ and } [\gamma]_\tau \to \gamma 
        \\
        \qquad \mbox{in  $ L^q(0, T; X) $, for any $ q \in [1, \infty) $,}
        \\
        \qquad \mbox{weakly-$ * $ in $ L^\infty(0, T; X) $,}
    \end{cases} \mbox{as $ \tau \downarrow 0 $,}
\end{equation*}
\begin{equation*}
    \begin{array}{c}
        [\overline{\gamma}]_\tau(t) \to \gamma(t), ~ [\underline{\gamma}]_\tau(t) \to \gamma(t), \mbox{ and } [\gamma]_\tau(t) \to \gamma(t)
        \\[1ex]
        \mbox{in $ X $, a.e. $ t \in (0, T) $, as $ \tau \downarrow 0 $.}
    \end{array}
\end{equation*}
\item[~~~~$ \bullet $]If $ \gamma \in W^{1, \infty}(Q) $, and the sequence $ \{ \gamma_i \}_{i = 0}^\infty \subset W^{1, \infty}(\Omega) $ is given as:
\begin{equation*}
    \gamma_i := \begin{cases}
        \gamma(t_i) \mbox{ in $ W^{1, \infty}(\Omega) $, if $ t_i \leq T $, } 
        \\
        \gamma(t_{i -1}) \mbox{ in $ W^{1, \infty}(\Omega) $, if $ t_{i -1} \leq T < t_i $,}
        \\
        0 \mbox{ in $ W^{1, \infty}(\Omega) $,  otherwise,}
    \end{cases}
    i = 0, 1, 2, \dots,
\end{equation*}
then
\begin{equation*}
    \begin{cases}
        \ds \sup_{\tau \in (0, 1)} \bigl\{ |[\overline{\gamma}]_\tau|_{L^\infty(Q)}, |[\underline{\gamma}]_\tau|_{L^\infty(Q)}, |[\gamma]_\tau|_{C(\overline{Q})} \bigr\} \leq |\gamma|_{C(\overline{Q})},
        \\
        \ds \sup_{\tau \in (0, 1)} \bigl\{ |\nabla [\overline{\gamma}]_\tau|_{L^\infty(Q)^N}, |\nabla [\underline{\gamma}]_\tau|_{L^\infty(Q)^N}, |\nabla [\gamma]_\tau|_{L^\infty(Q)^N} \bigr\} \leq |\nabla \gamma|_{L^\infty(Q)^N},
        \\
        \ds \sup_{\tau \in (0, 1)} |\partial_t [\gamma]_\tau|_{L^\infty(Q)} \leq |\partial_t \gamma|_{L^\infty(Q)},
    \end{cases}
\end{equation*}
\begin{equation*}
\begin{cases}
[\overline{\gamma}]_\tau \to \gamma \mbox{ and } [\underline{\gamma}]_\tau \to \gamma, \mbox{ in $ L^\infty(0, T; C(\overline{\Omega})) $,}
\\[1ex]
[\gamma]_\tau \to \gamma \mbox{ in $ C(\overline{Q}) $,}
\end{cases}
\end{equation*}
\begin{equation*}
\begin{cases}
\partial_t [\gamma]_\tau \to \partial_t \gamma \mbox{ weakly-$*$ in $ L^\infty(Q) $,}
\\
\qquad \mbox{and in the pointwise sense a.e. in $ Q $,}
\end{cases}
\end{equation*}
and
\begin{equation*}
\begin{cases}
    \nabla [\overline{\gamma}]_\tau \to \nabla \gamma, ~ \nabla [\underline{\gamma}]_\tau \to \nabla \gamma, \mbox{ and } \nabla [\gamma]_\tau \to \nabla \gamma
\\
\qquad \mbox{weakly-$*$ in $ L^\infty(Q)^N $,}
\\
\qquad \mbox{and in the pointwise sense a.e. in $ Q $,}
\end{cases}
\mbox{as $ \tau \downarrow 0 $.}
\end{equation*}
\end{description}
\end{remark}

\section{Main Theorems}\label{02MainThms}

In this paper, the main results are discussed under the following assumptions.
\begin{enumerate}
    \item[\hypertarget{A1}{(A1)}]The sextet of given data $ [a, b, \mu, \lambda, \omega, A]$ belongs to a class $ \mathscr{S}_0 \subset [\scrH]^6 $, defined as follows.
    \begin{equation*}
        \mathscr{S}_0 := \left\{ 
        \begin{array}{l|l} 
            [\tilde{a}, \tilde{b}, \tilde{\mu}, \tilde{\lambda}, \tilde{\omega}, \tilde{A}] \in [\scrH]^6 & \hspace{-2ex} 
            \parbox{7.7cm}{
            \begin{itemize}
                \item $ \tilde{a} \in W^{1, \infty}(Q) $ and $ \log \tilde{a} \in L^\infty(Q) $;
                \item $ \tilde{b} \in L^\infty(Q) $;
                \item $ \tilde{\mu} \in L^\infty(0, T; H) $ and  $ \tilde{\mu} \geq 0 $ a.e. in $ Q $;
                \item $ \tilde{\lambda} \in L^\infty(Q) $; ~~ $ \bullet $ $ \tilde{\omega} \in L^\infty(Q)^N $;
                \item $ \tilde{A} \in L^\infty(Q)^{N \times N} $, and the value $ \tilde{A}(t, x) \in \bbR^{N \times N} $ is positive and symmetric, for a.e. $ (t, x) \in Q $.
            \end{itemize}
        } \end{array} \right\};
    \end{equation*}
\item[\hypertarget{A2}{(A2)}]$ \nu > 0 $ is a fixed constant, and $ [h, k] \in \mathscr{V}^* \times \mathscr{V}_0^* $ and $ [p_0, z_0]  \in [H]^2 $ are given pairs of functions.
\end{enumerate}
\begin{remark}\label{Rem.relat_P-S}
    Let us define:
\begin{equation*}
    \delta_*(\tilde{a}) := \mathrm{ess} \hspace{0.5ex} \inf_{\hspace{-3.5ex}(t, x) \in Q} \tilde{a}(t, x) \in \bbR, \mbox{ for any $ \tilde{a} \in L^\infty(Q) $.}
\end{equation*}
Then, the assumption (\hyperlink{A1}{A1}) implies that 
\begin{equation*}
a(t, x) \geq \delta_*(a) > 0, \mbox{ for a.e. $ (t, x) \in Q $.} 
\end{equation*}
In view of the previous works \cite{MR2469586,MR2548486,MR2836555,MR2668289,MR3038131,MR3362773,MR3082861,MR3203495,MR3670006,Nakayashiki18} and the relations \eqref{relat_P-S01}--\eqref{relat_P-S02}, we can say that (\hyperlink{A1}{A1}) is a meaningful assumption for the sextet of data $ [a, b, \mu, \lambda, \omega, A] $, and especially for the components $ \mu $, $ \omega $, and $ A $. 
\end{remark}
\begin{remark}\label{Rem.sol}
We here verify the following key-properties.
\begin{enumerate}
    \item[\hypertarget{R1}{(R1)}]Since $ N \in \{1, 2, 3\} $, and $ \Omega \subset \bbR^N $ is a bounded domain, we have the continuous (compact) embedding $ V \subset L^4(\Omega) $. Therefore, for any $ 0 \leq \mu^\circ \in H $ and any $ p^\circ \in V $, we can say that $ \sqrt{\mu^\circ} p^\circ \in H $ and $ \mu^\circ p^\circ \in V^* $, via the following estimate:
\begin{align*}
\int_\Omega \mu^\circ |p^\circ| |\varphi^\circ| \, dx & \leq |\mu^\circ|_H |p^\circ|_{L^4(\Omega)}|\varphi^\circ|_{L^4(\Omega)}
\nonumber
\\
& \leq (C_V^{L^4})^2 |\mu^\circ|_H |p^\circ|_V |\varphi^\circ|_{V}, \mbox{ for any $ \varphi^\circ \in V $,}
\end{align*}
and the variational identity:
\begin{equation*}
\bigl\langle \mu^\circ p^\circ, \varphi^\circ \bigr\rangle_{V} := \bigl( \sqrt{\mu^\circ} p^\circ, \sqrt{\mu^\circ} \varphi^\circ \bigr)_{H}, \mbox{ for any $ \varphi^\circ \in V $, respectively,}
\end{equation*}
where $ C_V^{L^4} > 0 $ is the constant of embedding $ V \subset L^4(\Omega) $. Similarly, for any $ 0 \leq  \mu \in L^\infty(0, T; H) $ and any $ p \in \mathscr{V} $, we can say that $ \sqrt{\mu} p \in \scrH $ and $ \mu p \in \mathscr{V}^* $, via:
\begin{align*}
\int_0^T \hspace{-1ex} \int_\Omega & \mu |p| |\varphi| \, dx dt \leq \int_0^T |\mu(t)|_H |p(t)|_{L^4(\Omega)}|\varphi(t)|_{L^4(\Omega)} \, dt
\\
& \leq (C_V^{L^4})^2 |\mu|_{L^\infty(0, T; H)} \int_0^T |p(t)|_V |\varphi(t)|_V \, dt
\\
& \leq  (C_V^{L^4})^2 |\mu|_{L^\infty(0, T; H)} |p|_{\mathscr{V}} |\varphi|_{\mathscr{V}}, \mbox{ for any $ \varphi \in \mathscr{V} $,}
\end{align*}
and 
\begin{equation*}
\langle \mu p, \varphi \rangle_{\mathscr{V}} := ( \sqrt{\mu} p, \sqrt{\mu} \varphi )_{\scrH}, \mbox{ for any $ \varphi \in \mathscr{V} $, respectively.}
\end{equation*}
\item[\hypertarget{R2}{(R2)}]For any $ \tilde{a} \in W^{1, \infty}(Q) \cup L^\infty(0, T; W^{1, \infty}(\Omega)) $ and any $ w \in \mathscr{V}_0^* $, we can say that $ \tilde{a} w $ $ (= w \tilde{a}) $ $\in \mathscr{V}_0^* $, via the following variational form:
    \begin{equation*}
        \langle \tilde{a} w, \psi \rangle_{\mathscr{V}_0} := \langle w, \tilde{a} \psi \rangle_{\mathscr{V}_0}, \mbox{ for any $ \psi \in \mathscr{V}_0 $,}
\end{equation*}
and can estimate that:
\begin{equation*}
    |\tilde{a} w|_{\mathscr{V}_0^*} \leq (1 +C_{V_0}^{H}) \bigl( |\tilde{a}|_{L^\infty(Q)} +|\nabla \tilde{a}|_{L^\infty(Q)^N} \bigr) |w|_{\mathscr{V}_0^*},
\end{equation*}
        by using the constant $ C_{V_0}^{H} $ of the embedding $ V_0 \subset H $. In particular, if $ \tilde{a} \in W^{1, \infty}(Q) $ and $ w \in W^{1, 2}(0, T; V_0^*) $, then $ \tilde{a} w \in W^{1, 2}(0, T; V_0^*) $, and 
\begin{equation*}
\partial_t (\tilde{a} w) = \tilde{a} \partial_t w +w \partial_t \tilde{a} \mbox{ in $ \mathscr{V}_0^* $.}
\end{equation*}
        Meanwhile, if $ \tilde{a} \in W^{1, \infty}(Q) \cup L^\infty(0, T; W^{1, \infty}(\Omega)) $ and $ \log \tilde{a} \in L^\infty(Q) $, then  it is estimated that:
\begin{equation*}
    |\tilde{a} w|_{\mathscr{V}_0^*} \geq \frac{\delta_*(\tilde{a})^2}{(1 +C_{V_0}^{H})(\delta_*(\tilde{a}) +|\nabla \tilde{a}|_{L^{\infty}(Q)^N})} |w|_{\mathscr{V}_0^*}. 
\end{equation*}
\end{enumerate}
\end{remark}

Now, the solution to our system (\hyperlink{P}{P}) is defined as follows.
\begin{definition}\label{Def.solOf(P)}
A pair of functions $ [p, z] \in [\scrH]^2 $ is called a solution to (\hyperlink{P}{P}), iff. $ [p, z] $ fulfills the following conditions.
\begin{enumerate}
    \item[\hypertarget{P1}{(P1)}]$ p \in W^{1, 2}(0, T; V^*) \cap L^2(0, T; V) \subset C([0, T]; H) $ with $ p(0) = p_0 $ in $ H $;
\\
$ z \in W^{1, 2}(0, T; V_0^*) \cap L^2(0, T; V_0) \subset C([0, T]; H) $ with $ z(0) = z_0 $ in $ H $.
\item[\hypertarget{P2}{(P2)}]$ p $ solves the following variational identity:
\begin{equation*}
\begin{array}{c}
\ds \langle \partial_t p(t), \varphi \rangle_V +\int_\Omega \nabla p(t) \cdot \nabla \varphi \, dx +\int_\Omega  \mu(t) p(t) \varphi \, dx
\\[2ex]
+\bigl( \lambda(t) p(t) +\omega(t) \cdot \nabla z(t), \varphi \bigr)_{H} = \langle h(t), \varphi \rangle_{V},
\\[2ex]
\mbox{ for any $ \varphi \in V $, and a.e. $ t \in (0, T) $.}
\end{array}
\end{equation*}
\item[\hypertarget{P3}{(P3)}]$ z $ solves the following variational identity:
\begin{equation*}
\begin{array}{l}
    \ds \langle \partial_t z(t), a(t) \psi \rangle_{V_0} +( b(t)z(t), \psi )_{H}
\\[1ex]
\ds \qquad +\int_\Omega  \bigl( A(t) \nabla z(t) +\nu \nabla z(t) \bigr) \cdot \nabla \psi \, dx  +\int_\Omega p(t) \omega(t) \cdot \nabla \psi \, dx
\\[2ex]
\qquad = \langle k(t), \psi \rangle_{V_0}, \mbox{ for any $ \psi \in V_0 $, and a.e. $ t \in (0, T) $.}
\end{array}
\end{equation*}
\end{enumerate}
\end{definition}

\begin{remark}\label{Rem.sol2}
    Taking into account the notations as in \eqref{dualMapF} and \eqref{div_V0*}, and Remark \ref{Rem.sol}, the variational identities as in Definition \ref{Def.solOf(P)} (\hyperlink{P2}{P2}) and (\hyperlink{P3}{P3}) can be reformulated as follows.
\begin{equation*}
\begin{cases}
\partial_t p +(Fp -p) +\mu p +\lambda p +\omega \cdot \nabla z = h \mbox{ in $ \mathscr{V}^* $,}
\\[1ex]
a \partial_t z +bz -\mathrm{div} \, \bigl( A \nabla z +\nu \nabla z +p \omega \bigr) = k \mbox{ in $ \mathscr{V}_0^* $.}
\end{cases}
\end{equation*}
\end{remark}

In this paper, the solution to (\hyperlink{P}{P}) will be obtained by means of the time-discretization method. In view of this, we denote by $ \tau \in (0, 1) $ the constant of time step-size, and consider the following semi-discrete (time discrete, spatially continuous) scheme, denoted by (\hyperlink{DP}{DP})$_\tau$, which is a discrete version of (\hyperlink{P}{P}). 
\begin{align}\label{DP_1st}
    \parbox{7ex}{(\hypertarget{DP}{DP})$_\tau$:} &  
\nonumber
\\
& \parbox{12cm}{
$ \begin{array}{l}
\ds 
\frac{1}{\tau} \bigl( p_i -p_{i -1}, \varphi \bigr)_H +\ds \int_\Omega \nabla p_i \cdot \nabla \varphi \, dx +\int_\Omega  \mu_i p_i \varphi \, dx
\\[2ex]
    \qquad +\ds \bigl( \lambda_{i} p_i +\omega_i \cdot \nabla z_i, \varphi \bigr)_{H} = \langle h_i, \varphi \rangle_V, 
\\[2ex]
\qquad \mbox{for every $ \varphi \in V $, $ i = 1, 2, 3, \dots $, with the initial data $ p_0 \in H $,}  
\end{array} $
}
\end{align}
\vspace{-5.5ex}
\begin{align}\label{DP_2nd}
\parbox{7ex}{~~} &  
\nonumber
\\
& \parbox{12cm}{
$ \begin{array}{l}
\ds \frac{1}{\tau} \bigl( a_{i} (z_i - z_{i -1}), \psi \bigr)_{H} +(b_{i} z_i, \psi )_{H}
\\[2ex]
\qquad +\ds \int_\Omega  \bigl( A_i \nabla z_i +\nu \nabla z_i \bigr) \cdot \nabla \psi \, dx +\int_\Omega p_i \omega_i \cdot \nabla \psi \, dx
    = \langle k_i, \psi \rangle_{V_0}, 
\\[2ex]
\qquad \mbox{for every $ \psi \in V_0 $, $ i = 1, 2, 3, \dots $, with the initial data $ z_0 \in H $.} 
\end{array}
 $
}
\end{align}
\noindent
The solution to (\hyperlink{DP}{DP})$_\tau$ is defined as a sequence $ \{ [p_i, z_i] \}_{i = 0}^\infty \subset [H]^2 $, such that
\begin{equation*}
    [p_i, z_i] \in V \times V_0 \mbox{ for all $ i \geq 1 $, and $ \{ [p_i, z_i] \}_{i = 0}^\infty $  fulfills \eqref{DP_1st} and \eqref{DP_2nd}.}
\end{equation*}
We let $ \{ [a_{i}, b_{i}, \mu_i, \lambda_{i}, \omega_i, A_i] \}_{i = 0}^\infty  $ be a bounded sequence in $ W^{1, \infty}(\Omega) \times  L^\infty(\Omega) \times H \times  L^\infty(\Omega)  \times L^\infty(\Omega)^{N} \times L^\infty(\Omega)^{N \times N} $, such that:
\begin{subequations}\label{D-given}
    \begin{equation}\label{D-given_a}
        \begin{cases}
            \ds \sup_{i \geq 0} |a_{i}|_{W^{1, \infty}(\Omega)} \leq |a|_{W^{1, \infty}(Q)}, ~ 
            \ds \sup_{i \geq 0} |b_{i}|_{L^\infty(\Omega)} \leq |b|_{L^\infty(Q)}, ~ 
            \\
            \ds \sup_{i \geq 0} |\mu_i|_H \leq |\mu|_{L^\infty(0, T; H)}, ~ 
            \ds \sup_{i \geq 0} |\lambda_{i}|_{L^\infty(\Omega)} \leq |\lambda|_{L^\infty(Q)}, ~ 
            \\
            \ds \sup_{i \geq 0} |\omega_i|_{L^\infty(\Omega)^N} \leq |\omega|_{L^\infty(Q)^N}, ~ 
            \ds \sup_{i \geq 0} |A_i|_{L^\infty(\Omega)^{N \times N}} \leq |A|_{L^\infty(Q)^{N \times N}},
        \end{cases}
    \end{equation} 
    \begin{equation}\label{D-given_b}
        \begin{array}{c}
            \mbox{$ a_{i} \geq \delta_*(a) $,  $ \mu_i \geq 0 $, and $ A_i $ is positive and symmetric,} 
            \\
            \qquad\mbox{a.e. in $ \Omega $, for $ i = 0, 1, 2, \dots $,}
        \end{array}
    \end{equation} 
    \begin{equation}\label{D-given_c}
        \left\{ \parbox{7cm}{
            $ [\overline{a}]_\tau \to a $ in $ L^\infty(0, T; C(\overline{\Omega})) $, 
            \\[1ex]
            $ [{a}]_\tau \to a $  in $ C(\overline{Q}) $, 
        } \right. 
    \end{equation}
    \begin{equation}\label{D-given_e}
        \left\{ \parbox{7cm}{
            $ [\overline{\mu}]_\tau \to \mu $ weakly-$ * $ in $ L^\infty(0, T; H) $, 
            \\[1ex]
            $ [\overline{\mu}]_\tau(t) \to \mu(t) $ in $ H $, a.e. $ t \in (0, T) $, 
        }\right. 
    \end{equation}
    \begin{align}\label{D-given_o}
        \bigl[ \partial_t & [a]_\tau, \nabla [\overline{a}]_\tau, [\overline{b}]_\tau, [\overline{\lambda}]_\tau, [\overline{\omega}]_\tau, [\overline{A}]_\tau \bigr]
        \to [\partial_t a, \nabla a, b,  \lambda, \omega, A] \mbox{ weakly-$*$ }
        \nonumber
        \\
        & \mbox{in $ L^\infty(Q) \times L^\infty(Q)^N \times L^\infty(Q) \times L^\infty(Q) \times L^\infty(Q)^N \times L^\infty(Q)^{N \times N} $,}
        \nonumber
        \\ 
        & \mbox{and in the pointwise sense a.e. in $ Q $, as $ \tau \downarrow 0 $,}
    \end{align}
and $ \{ [h_i, k_i] \}_{i = 0}^\infty $ is a bounded sequence in $ \mathscr{V}^* \times \mathscr{V}_0^* $, such that:
\begin{equation}\label{rxForces01}
    \begin{cases}
        \ds K^* := \sup_{\tau \in (0, 1)} \bigl| \bigl[ [\overline{h}]_\tau, [\overline{k}]_\tau \bigr] \bigr|_{\mathscr{V}^* \times \mathscr{V}_0^*} < \infty,
        \\
        \bigl[ [\overline{h}]_\tau, [\overline{k}]_\tau \bigl] \to [h, k] \mbox{ in $ \mathscr{V}^* \times \mathscr{V}_0^* $, as $ \tau \downarrow 0 $.}
    \end{cases}
\end{equation}
\end{subequations}
\begin{remark}\label{Rem.HA01}
    Notice that it is straightforward to obtain $ \{ [a_{i}, b_{i}, \mu_i, \lambda_{i}, \omega_i, A_i] \}_{i = 0}^\infty  $ and \linebreak $ \{ [h_i, k_i] \}_{i = 0}^\infty $ fulfilling \eqref{D-given}, because the assumptions (\hyperlink{A1}{A1}) and (\hyperlink{A2}{A2}) allow us to apply the standard method as in Remark \ref{Rem.t-discrete} (\hyperlink{Fact1}{Fact\,1}). 
\end{remark}

We are now ready to state the Main Theorems of this paper. 

\hypertarget{mainThrm01}{}
\begin{mainThrm}[Well-posedness]\label{mainThrm1}
\begin{description}
    \item[\textmd{\textit{\hypertarget{I-A}{(I-A)} (Existence and uniqueness)}}]Under the as- \linebreak sumptions (\hyperlink{A1}{A1}) and (\hyperlink{A2}{A2}), the system (\hyperlink{P}{P}) admits a unique solution $ [p, z] $. 
    \item[\textmd{\textit{\hypertarget{I-B}{(I-B)} (Continuous dependence on data)}}]For every forcing pairs $ [h^\ell, k^\ell] \in \mathscr{V}^* \times \mathscr{V}_0^* $, initial pairs $ [p_0^\ell, z_0^\ell] \in [H]^2 $, and sextets of data $ [a^\ell, b^\ell, \mu^\ell, \lambda^\ell, \omega^\ell, A^\ell] \in \mathscr{S}_0 $, $ \ell = 1, 2 $, let $ [p^\ell, z^\ell] \in [\scrH]^2 $, $ \ell = 1, 2 $, be the corresponding solutions to (\hyperlink{P}{P}). Additionally, let $ C^* = C^*(a^1, b^1, \lambda^1, \omega^1) $ be a positive constant, depending on $ |a^1|_{W^{1, \infty}(Q)} $, $ \delta_*(a^1) $, $ |b^1|_{L^\infty(Q)} $, $ |\lambda^1|_{L^\infty(Q)} $, and $ |\omega^1|_{L^\infty(Q)^N} $, which is defined as:
\begin{align}\label{defOf_C*}
    C^* := & \ds \frac{9 (1 +\nu)}{\min \{1, \nu, \delta_*(a^1) \}} \bigl( 1 +(C_V^{L^4})^2 +(C_V^{L^4})^4 +(C_{V_0}^{L^4})^2 \bigr) \cdot
\nonumber
    \\
    & \cdot \bigl( 1 +|a^1|_{W^{1, \infty}(Q)} +|b^1|_{L^\infty(Q)} +|\lambda^1|_{L^\infty(Q)} +|\omega^1|_{L^\infty(Q)^N}^2 \bigr),
\end{align}
with the constants $ C_V^{L^4} > 0 $ and $ C_{V_0}^{L^4}  > 0 $ of the respective embeddings $ V \subset L^4(\Omega) $ and $ V_0 \subset L^4(\Omega) $. Then, it holds that:
\begin{align}\label{contiDep}
    |(p^1 & -p^2)(t)|_H^2 +| \mbox{\small$ \sqrt{a^1(t)} $}(z^1 -z^2)(t)|_H^2
\nonumber
\\[1ex]
& +\int_0^t |(p^1 -p^2)(\varsigma)|_V^2 \, d \varsigma +\nu \int_0^t |(z^1 -z^2)(\varsigma)|_{V_0}^2 \, d \varsigma
\nonumber
\\[1ex]
    \leq & e^{3C^* T} \bigl( |p_0^1 -p_0^2|_H^2 +| \mbox{\small$ \sqrt{a^1(0)} $}(z_0^1 -z_0^2)|_H^2 \bigr) 
\nonumber
\\[1ex]
    & +2 C^* e^{3 C^* T} \int_0^T \bigl( |(h^1 -h^2)(t)|_{V^*}^2 +|(k^1 -k^2)(t)|_{V_0^*}^2 \bigr) \, dt 
\\[1ex]
    & +2 C^* e^{3 C^* T} \int_0^T R^*(t) \, dt, \mbox{ for all $ t \in [0, T] $,}
\nonumber
\end{align}
where 
\begin{align}\label{defOf_R*}
    & R^*(t) := |\partial_t z^2(t)|_{V_0^*}^2 \bigl( |a^1 -a^2|_{C(\overline{Q})}^2 +|\nabla (a^1 -a^2)(t)|_{L^4(\Omega)^N}^2 \bigr)
    \nonumber
    \\[1ex]
    & \quad +|p^2(t)|_{V}^2 \bigl( |(\mu^1 -\mu^2)(t)|_{H}^2 +|(\omega^1 -\omega^2)(t)|_{L^4(\Omega)^N}^2 \bigr) 
    \nonumber
    \\[1ex]
    & \quad +|z^2(t)|_{V_0}^2 |(b^1 -b^2)(t)|_{L^4(\Omega)}^2 +|p^2(t)(\lambda^1 -\lambda^2)(t)|_H^2
    \nonumber
    \\[1ex]
    & \quad +|\nabla z^2(t) \cdot (\omega^1 -\omega^2)(t)|_H^2 +|(A^1 -A^2)(t) \nabla z^2(t)|_{[H]^N}^2, 
    \\[1ex]
    & \mbox{for  a.e. $ t \in (0, T) $.}
    \nonumber
\end{align}
\end{description}
\end{mainThrm}

\begin{mainThrm}[Numerical scheme and convergence]\label{mainThrm2}
Under the assumptions (\hyperlink{A1}{A1}) \linebreak and (\hyperlink{A2}{A2}), the following two items hold. 
\begin{description}
    \item[\textmd{\textit{\hypertarget{II-A}{(II-A)}}}]There is a constant $ \tau_* = \tau_*(a, b, \lambda, \omega) \in (0, 1) $, depending on $ \delta_*(a) $, $ |b|_{L^\infty(Q)} $, $ |\lambda|_{L^\infty(Q)} $, and $ |\omega|_{L^\infty(Q)^N} $, such that for any $ \tau \in (0, \tau_*) $, the semi-discrete scheme (\hyperlink{DP}{DP})$_\tau$ admits a unique solution $ \{ [p_i, z_i] \}_{i = 0}^\infty \subset [H]^2 $.
\item[\textmd{\textit{\hypertarget{II-B}{(II-B)}}}]The sequence of the linear time-interpolation $ \bigl\{ \bigl[ [p]_\tau, [z]_\tau] \bigr] \bigr\}_{\tau \in (0, \tau_*)} $ converges to the solution $ [p, z] $ to (\hyperlink{P}{P}), as $ \tau \downarrow 0 $, in the sense that: 
    \begin{subequations}\label{concluConv}
    \begin{align}\label{concluConv_a}
        \bigl[ [p]_\tau, & [z]_\tau \bigr] \to [p, z] \mbox{ in $ [\scrH]^2 $, in $ C([0, T]; V^*) \times C([0, T]; V_0^*) $, }
        \nonumber
        \\
        & \mbox{weakly in $ W^{1, 2}(0, T; V^*) \times W^{1, 2}(0, T; V_0^*) $,}
        \nonumber
        \\
        & \mbox{and weakly-$*$ in $ L^\infty(0, T; H)^2 $, as $ \tau \downarrow 0 $.} 
    \end{align}
    Moreover, if $ [p_0, z_0] \in V \times V_0 $, then:
    \begin{equation}\label{concluConv_c}
        \bigl[ [{p}]_\tau, [{z}]_\tau \bigr] \to [p, z] \mbox{ in $ C([0, T]; H)^2 $, and in $ \mathscr{V} \times \mathscr{V}_0 $, as $ \tau \downarrow 0 $.}
    \end{equation}
\end{subequations}
\end{description}
\end{mainThrm}
\begin{remark} \label{Rem.mainTh2}
    Notice that the convergence \eqref{concluConv_a} implies that:
    \begin{align*}
        \bigl[ [\overline{p}]_\tau, [\overline{z}]_\tau \bigr] & \to [p, z] \mbox{ and } \bigl[ [\underline{p}]_\tau, [\underline{z}]_\tau \bigr] \to [p, z],  
        \\
        & 
        \mbox{in $ [\scrH]^2 $ and in $ L^\infty(0, T; V^*) \times L^\infty(0, T; V_0^*) $, as $ \tau \downarrow 0 $.}
    \end{align*}
    Also, the convergences \eqref{concluConv_c} implies that:
    \begin{align*}
        \bigl[ [\overline{p}]_\tau, [\overline{z}]_\tau \bigr] & \to [p, z] \mbox{ and } \bigl[ [\underline{p}]_\tau, [\underline{z}]_\tau \bigr] \to [p, z], 
        \mbox{ in $ L^\infty(0, T; H)^2 $, as $ \tau \downarrow 0 $.}
    \end{align*}
    Moreover, when $ [p_0, z_0] \in V \times V_0 $, we will obtain that:
    \begin{align*}
        \bigl[ [\overline{p}]_\tau, [\overline{z}]_\tau \bigr] & \to [p, z] \mbox{ and } \bigl[ [\underline{p}]_\tau, [\underline{z}]_\tau \bigr] \to [p, z], 
        \mbox{ in $ \mathscr{V} \times \mathscr{V}_0 $, as $ \tau \downarrow 0 $.}
    \end{align*}
    in the process of the proof (cf. \eqref{mThIIB-11}--\eqref{mThIIB-14}). 
\end{remark}
\section{Key-Lemmas}\label{03Key-Lems}

In this Section, we prove a few Key-Lemmas that are vital for our Main Theorems. 
\begin{keyLemma}\label{axLem01}
    Let the assumptions (\hyperlink{A1}{A1}) and (\hyperlink{A2}{A2}) hold. Let $ a^\circ \in L^\infty(\Omega) $, $ b^\circ \in L^\infty(\Omega) $, $ \mu^\circ \in H $, $ \lambda^\circ \in L^\infty(\Omega) $, $ \omega^\circ \in L^\infty(\Omega)^N $, and $ A^\circ \in L^\infty(\Omega)^{N \times N} $ be functions, such that
    \begin{subequations}\label{axLem01-01}
        \begin{equation} \label{axLem01-01_a}
            \begin{cases}
                \ds |a^\circ|_{L^\infty(\Omega)} \leq |a|_{L^\infty(Q)}, ~  |b^\circ|_{L^\infty(\Omega)} \leq |b|_{L^\infty(Q)},
                \\
                |\mu^\circ|_H \leq |\mu|_{L^\infty(0, T; H)}, ~ |\lambda^\circ|_{L^\infty(\Omega)} \leq |\lambda|_{L^\infty(Q)} ,
                \\
                |\omega^\circ|_{L^\infty(\Omega)^N} \leq |\omega|_{L^\infty(Q)^N}, \mbox{ and } |A^\circ|_{L^\infty(\Omega)^{N \times N}} \leq |A|_{L^\infty(Q)^{N \times N}},
            \end{cases}
        \end{equation}
        and
        \begin{equation} \label{axLem01-01_b}
            a^\circ \geq \delta_*(a), ~ \mu^\circ \geq 0, \mbox{ and $ A^\circ $ is positive and symmetric, a.e. in $ \Omega $.}
        \end{equation}
    \end{subequations}
    Then, there is a small constant $ \tau_0 = \tau_0(a, b, \lambda, \omega) \in (0, 1) $, depending on $ \delta_*(a) $, $ |b|_{L^\infty(Q)} $, $ |\lambda|_{L^\infty(Q)} $, and  $ |\omega|_{L^\infty(Q)^N} $, such that for every pairs of functions $ [h^\circ, k^\circ] \in V^* \times V_0^* $ and $ [p_0, z_0] \in [H]^2 $, the following variational system admits a unique solution $ [p, z] \in V \times V_0 $: 
\begin{equation}\label{ax1st}
\begin{array}{l}
\ds \frac{1}{\tau} ( p -p_0, \varphi )_H +\int_\Omega \nabla p \cdot \nabla \varphi \, dx +\int_\Omega \mu^\circ p \varphi \, dx
\\[2ex]
    \qquad \ds +\bigl( \lambda^\circ p +\omega^\circ \cdot \nabla z, \varphi \bigr)_{H} = \langle h^\circ, \varphi \rangle_V, \mbox{ for any $ \varphi \in V $,}
\end{array}
\end{equation}
\begin{equation}\label{ax2nd}
\begin{array}{l}
\ds \frac{1}{\tau} ( a^\circ(z -z_0), \psi )_H +( b^\circ z, \psi )_H +\int_\Omega \bigl( A^\circ \nabla z +\nu \nabla z \bigr) \cdot \nabla \psi \, dx
\\[2ex]
    \qquad \ds +\int_\Omega p \omega^\circ \cdot \nabla \psi \, dx = \langle k^\circ, \psi \rangle_{V_0}, \mbox{ for any $ \psi \in V_0 $.}
\end{array}
\end{equation}
\end{keyLemma}
\begin{proof}
First, for the proof of existence, we  define a (non-convex) functional $ \mathcal{E} : [H]^2 \longrightarrow (-\infty, \infty] $, by letting:
    \begin{equation}\label{enrgyE}
[p, z] \in [H]^2 \mapsto \mathcal{E}(p, z) := \left\{ \begin{array}{ll}
\multicolumn{2}{l}{
\ds \frac{1}{2 \tau} \bigl( |p -p_0|_H^2 +|\sqrt{a^\circ}(z -z_0)|_H^2 \bigr)
}
\\[2ex]
& \ds +\frac{1}{2} \int_\Omega \bigl( |\nabla p|^2 +|[A^\circ]^{\frac{1}{2}} \nabla z|^2 +\nu |\nabla z|^2 \bigr) \, dx 
\\[2ex]
& \ds +\int_\Omega |\sqrt{\mu^\circ} p|^2 \, dx +\int_\Omega p \bigl( \omega^\circ \cdot \nabla z \bigr) \, dx 
\\[2ex]
& \ds +\frac{1}{2} \int_\Omega \bigl( \lambda^\circ |p|^2 +b^\circ |z|^2 \bigr) \, dx 
\\[2ex]
    & -\langle h^\circ, p \rangle_V -\langle k^\circ, z \rangle_{V_0}, \mbox{ if $ [p, z] \in V \times V_0 $,}
\ \\
\ \\[-0ex]
\infty, & \mbox{otherwise,}
\end{array} \right.
\end{equation}
and we set:
\begin{equation}\label{axLem01-03}
    \tau_0 = \tau_0(a, b, \lambda, \omega) := \frac{\min \{1, \nu,  \delta_*(a) \}}{16(1 +\nu +\delta_*(a))(1 +|b|_{L^\infty(Q)} +|\lambda|_{L^\infty(Q)} +|\omega|_{L^\infty(Q)^N}^2)}.
\end{equation}
    Notice that nonconvexity in $ \mathcal{E} $ is due to $\int_\Omega p \bigl( \omega^\circ \cdot \nabla z \bigr) \, dx +\frac{1}{2} \int_\Omega \bigl( \lambda^\circ |p|^2 +b^\circ |z|^2 \bigr) \, dx $. Then, in the light of Remark \ref{Rem.sol}, it is easy to check that $ \mathcal{E} $ is a proper lower semi-continuous functional on $ [H]^2 $, such that:
\begin{align*}
    \mathcal{E}(p, z) & \geq \frac{1}{8 \tau} \bigl( |p|_H^2 +\delta_*(a) |z|_H^2 \bigr) +\frac{1}{4} \bigl( |\nabla p|_{[H]^N}^2 +\nu |\nabla z|_{[H]^N}^2 \bigr) 
\\
& \qquad -\frac{1}{2 \tau} \bigl( |p_0|_H^2 +\delta_*(a) |z_0|_H^2 \bigr) -\left( |h^\circ|_{V^*}^2 +\frac{2}{\nu} |k^\circ|_{V_0^*}^2 \right),
\\
& \mbox{for any $ [p, z] \in V \times V_0 $, whenever $ \tau \in (0, \tau_0) $,}
\end{align*}
    via the following computations:
    \begin{align*}
        \frac{1}{2 \tau} \bigl( |p & -p_0|_H^2 +|\sqrt{a^\circ} (z -z_0)|_H^2 \bigr) 
        \\
        \geq & \frac{1}{4 \tau} \bigl( |p|_H^2 +\delta_*(a) |z|_H^2 \bigr) -\frac{1}{2 \tau} \bigl( |p_0|_H^2 +\delta_*(a) |z_0|_H^2 \bigr),
    \end{align*}
    \begin{align*}
        \int_\Omega p (\omega^\circ & \cdot \nabla z) \, dx +\frac{1}{2} \int_\Omega \bigl( \lambda^\circ |p|^2 +b^\circ |z|^2 \bigr) \, dx 
        \\
        \geq & -\frac{\nu}{8} |\nabla z|_{[H]^N}^2 -\left( \frac{1}{2} |\lambda|_{L^\infty(Q)} +\frac{2}{\nu} |\omega|_{L^\infty(Q)^N}^2 \right) |p|_H^2 -\frac{1}{2} |b|_{L^\infty(Q)} |z|_H^2, 
    \end{align*}
    and
    \begin{align*}
        \langle h^\circ, p \rangle_V + & \langle k^\circ, z \rangle_{V_0} \geq -\frac{1}{4} \bigl( |p|_H^2 +|\nabla p|_{[H]^N}^2 \bigr) -\frac{\nu}{8} |\nabla z|_{[H]^N}^2 -|h^\circ|_{V^*}^2 -\frac{2}{\nu} |k^\circ|_{V_0^*}^2.
    \end{align*}
    Additionally, when $ \tau \in (0, \tau_0) $, the system \{\eqref{ax1st},\eqref{ax2nd}\} coincides with the stationarity system for $ \min \mathcal{E} $, and hence, the solution to \{\eqref{ax1st},\eqref{ax2nd}\} is easily obtained, by means of the direct method of calculus of variations (cf. \cite[Theorem 3.2.1]{MR2192832}).
\bigskip

    Next, to prove uniqueness, we assume that there are two solutions $ [p^{\ell}, z^{\ell}] \in V \times V_0 $, $ \ell = 1, 2 $, to the system \{\eqref{ax1st},\eqref{ax2nd}\}. Besides, let us take the difference between the equations \eqref{ax1st} (resp. \eqref{ax2nd}) corresponding to $ p^\ell $ (resp. $ z^\ell $), $ \ell = 1, 2 $, and put $ \varphi = p^1 -p^2 $ (resp. $ \psi = z^1 -z^2 $). Then, taking the sum of the results, we arrive at
\begin{align*}
\frac{1}{\tau} & \bigl( |p^1 -p^2|_H^2 +|\sqrt{a^\circ}(z^1 -z^2)|_H^2 \bigr) +\int_\Omega b^\circ |z^1 -z^2|^2 \, dx
\\
& +\bigl( |\nabla (p^1 -p^2)|_{[H]^N}^2 +|[A^\circ]^{\frac{1}{2}}\nabla (z^1 -z^2)|_{[H]^N}^2 +\nu |\nabla (z^1 -z^2)|_{[H]^N}^2 \bigr)
\\
    & +|\sqrt{\mu^\circ} (p^1 -p^2)|_H^2  +\int_\Omega \lambda^\circ |p^1 -p^2|^2 \, dx 
+2 \int_\Omega (p^1 -p^2) \omega^\circ \cdot \nabla (z^1 -z^2) \, dx = 0. 
\end{align*}
Here, applying \eqref{axLem01-01}--\eqref{ax2nd}, \eqref{axLem01-03}, and Young's inequality, it is inferred that:
\begin{equation*}
\frac{1}{2 \tau} \bigl( |p^1 -p^2|_H^2 +\delta_*(a)|z^1 -z^2|_H^2 \bigr) \leq 0, \mbox{ whenever $ \tau \in (0, \tau_0) $.}
\end{equation*}
Since $\delta_*(a) > 0$ (cf. Remark \ref{Rem.relat_P-S}), the proof is finished.
\end{proof}

Next, let us define:
\begin{equation*}
\begin{cases}
\mathscr{X} :=  \bigl( \mathscr{V} \times W^{1, 2}(0, T; V^*) \bigr) \times \bigl( \mathscr{V}_0 \times W^{1, 2}(0, T; V_0^*) \bigr),
\\
\mathscr{Y} := L^2(0, T; V^*) \times L^2(0, T; V_0^*) ~ \bigl( = \mathscr{V}^* \times \mathscr{V}_0^* \bigr),
\end{cases}
\end{equation*}
and define a class of sextets of functions $ \mathscr{S} $ as follows: 
\begin{equation*}
\mathscr{S} := \left\{ \begin{array}{l|l}
    [\tilde{a}, \tilde{b}, \tilde{\mu}, \tilde{\lambda}, \tilde{\omega}, \tilde{A}] \in [\scrH]^6 & \parbox{6.4cm}{$ \tilde{a} \in W^{1, \infty}(Q) \cup L^\infty(0, T; W^{1, \infty}(\Omega)) $, and $ [\delta_*(\tilde{a}), \tilde{b}, \tilde{\mu}, \tilde{\lambda}, \tilde{\omega}, \tilde{A}] \in \mathscr{S}_0 $}
\end{array} \right\}.
\end{equation*}
Then, for any sextet of data $ [a, b, \mu, \lambda, \omega, A] \in \mathscr{S} $, we can define an operator $ \mathcal{T} = \mathcal{T}(a, b, \mu, \lambda, \omega, A): \mathscr{X} \longrightarrow \mathscr{Y} $, by letting:
\begin{align}\label{op_A}
    [p, \tilde{p}, z, \tilde{z}] \in \mathscr{X} & \mapsto \mathcal{T} [p, \tilde{p}, z, \tilde{z}] = \mathcal{T}(a, b, \mu, \lambda, \omega, A)[p, \tilde{p}, z, \tilde{z}]
    \nonumber
    \\[1ex]
    & := \rule{-1pt}{20pt}^\mathrm{t} \hspace{-0.5ex} 
    \left[ \begin{array}{c}
        \partial_t \tilde{p} +(Fp -p) +\mu p +\lambda p +\omega \cdot \nabla z 
        \\[1ex]
        a \partial_t \tilde{z} +bz -\mathrm{div} \, \bigl( A \nabla z +\nu \nabla z +p \omega \bigr)
    \end{array} \right] \in \mathscr{Y}. 
\end{align}
\begin{remark}\label{Rem.4concl1}
    For any fixed sextet of data $ [a, b, \mu, \lambda, \omega, A] \in \mathscr{S} $, we notice that the operator $ \mathcal{T} = \mathcal{T}(a, b, \mu, \lambda, \omega, A) : \mathscr{X} \longrightarrow \mathscr{Y} $ is bounded and linear, i.e. $ \mathcal{T}  = \mathcal{T}(a, b, \mu, \lambda, \omega, A) \in \mathscr{L}(\mathscr{X}; \mathscr{Y}) $. Indeed, invoking Remark \ref{Rem.sol}, we can calculate that:
    \begin{align*}
        \bigl| \mathcal{T}(a, & b, \mu, \lambda, \omega, A)[p, \tilde{p}, z, \tilde{z}] \bigr|_{\mathscr{Y}}
        \\
        \leq & \bigl| \partial_t \tilde{p} +(Fp -p) +\mu p +(\lambda p +\omega \cdot \nabla z) \bigr|_{\mathscr{V}^*} 
        \\
        & \quad +\bigl| a \partial_t \tilde{z} +bz -\mathrm{div}\, \bigl( A \nabla z +\nu \nabla z +p \omega \bigr) \bigr|_{\mathscr{V}_0^*}
        \\
        \leq & |\partial_t \tilde{p}|_{\mathscr{V}^*} +\bigl( 1 +(C_{V}^{L^4})^2|\mu|_{L^\infty(0, T; H)} +|\lambda|_{L^\infty(Q)} +|\omega|_{L^\infty(Q)^N} \bigr) |p|_{\mathscr{V}}
        \\
        & \quad +(1 +C_{V_0}^{H})(|a|_{L^\infty(Q)} +|\nabla a|_{L^\infty(Q)^N}) |\partial_t \tilde{z}|_{\mathscr{V}_0^*} 
        \\
        & \quad +\bigl( C_{V_0}^{H} |b|_{L^\infty(Q)} +|\omega|_{L^\infty(Q)^N} +|A|_{L^\infty(Q)^{N \times N}} +\nu \bigr) |z|_{\mathscr{V}_0}
        \\
        & \mbox{for any $ [p, \tilde{p}, z, \tilde{z}] \in \mathscr{X} $.}
    \end{align*}
    So, putting:
    \begin{align} \label{defOf_M0}
        M_0 = & M_0( a, b, \mu, \lambda, \omega, A) :=  2 (1 +\nu)(1 +(C_V^{L^4})^2 +C_{V_0}^H) \cdot
        \nonumber
        \\
        &  \cdot \left( \begin{array}{c} 
            1 +|a|_{L^\infty(Q)} +|\nabla a|_{L^\infty(Q)^N} +|b|_{L^\infty(Q)}  +|\mu|_{L^\infty(0, T; H)}
            \\[1ex]
            +|\lambda|_{L^\infty(Q)} +|\omega|_{L^\infty(Q)^N} +|A|_{L^\infty(Q)^{N \times N}} 
        \end{array} \right),
    \end{align}
    it is estimated that:
    \begin{equation}\label{estFinal00}
        \bigl| \mathcal{T}(a, b, \mu, \lambda, \omega, A)[p, \tilde{p}, z, \tilde{z}] \bigr|_{\mathscr{Y}} \leq M_0 \bigl| [p, \tilde{p}, z, \tilde{z}] \bigr|_{\mathscr{X}}, \mbox{ for any $ [p, \tilde{p}, z, \tilde{z}] \in \mathscr{X} $.}
    \end{equation}
    Similarly, putting:
    \begin{equation*}
        [h, k] := \mathcal{T}(a, b, \mu, \lambda, \omega, A)[p, \tilde{p}, z, \tilde{z}] \mbox{ in $ \mathscr{Y} $,} 
    \end{equation*}
    and
    \begin{align}\label{defOf_M1}
        M_1 = & M_1 (a, b, \mu, \lambda, \omega, A) 
        \nonumber
        \\
        := & 2 (1 +\nu)(1 +(C_V^{L^4})^2 +C_{V_0}^H)  \left( 1 +\frac{(1 +C_{V_0}^H)(|a|_{L^\infty(Q)} +|\nabla a|_{L^\infty(Q)^N})}{\delta_*(a)^2} \right)\cdot
        \nonumber
        \\
        &  \cdot \left( \begin{array}{c}
            1 +|b|_{L^\infty(Q)} +|\mu|_{L^\infty(0, T; H)} +|\lambda|_{L^\infty(Q)} +|\omega|_{L^\infty(Q)^N} +|A|_{L^\infty(Q)^{N \times N}} 
        \end{array} \right),  
    \end{align}
    we also see from Remark \ref{Rem.sol} (\hyperlink{R2}{R2}) that:
    \begin{align}\label{estFinal01}
        \bigl| [\partial_t \tilde{p}, \partial_t \tilde{z}] \bigr|_\mathscr{Y} \leq & \left( 1 +\frac{(1 +C_{V_0}^H)(\delta_*(a) +|\nabla a|_{L^\infty(Q)^N})}{\delta_*(a)^2} \right) \cdot 
        \nonumber
        \\
        & \cdot \left( \begin{array}{c}
        \bigl| (Fp -p) +\mu p +\lambda p +\omega \cdot \nabla z -h \bigr|_{\mathscr{V}^*}^2 
            \\[1ex]
            +\bigl| b z -\mathrm{div} \bigl( A \nabla z +\nu \nabla z +p \omega \bigr) -k \bigr|_{\mathscr{V}_0^*}^2 
    \end{array} \right)^{\hspace{-0.5ex}\frac{1}{2}}
        \\
        \leq & M_1 \bigl( |[h, k]|_\mathscr{Y} +|[p, z]|_{\mathscr{V} \times \mathscr{V}_0} \bigr).
        \nonumber
    \end{align}
\end{remark}
Now, the second auxiliary lemma is concerned with the continuous dependence of the value $ \mathcal{T} = \mathcal{T}(a, b, \mu, \lambda, \omega, A) [p, \tilde{p}, z, \tilde{z}] $ with respect to data $ [a, b, \mu, \lambda, \omega, A] \in \mathscr{S} $ and variable $ [p, \tilde{p}, z, \tilde{z}] $.

\begin{keyLemma}\label{axLem02}
    Let $ [a, b, \mu, \lambda, \omega, A] \in \mathscr{S} $ be any sextet of data, and let  $ \mathcal{T} = $ \linebreak $ \mathcal{T}(a, b, \mu, \lambda, \omega, A) \in \mathscr{L}(\mathscr{X}; \mathscr{Y}) $ be the operator as in \eqref{op_A}. Also, let us take a sequence \linebreak $ \{ [a^n, b^n, \mu^n, \lambda^n, \omega^n, A^n] \}_{n = 1}^\infty \subset \mathscr{S} $, and consider a sequence of operators:
\begin{center}
    $ \{ \mathcal{T}^n \}_{n = 1}^\infty := \bigl\{ \mathcal{T}(a^n, b^n, \mu^n, \lambda^n, \omega^n, A^n) \bigr\}_{n = 1}^\infty \subset \mathscr{L}(\mathscr{X}; \mathscr{Y}) $.
\end{center}
Also, let us assume that:
    \begin{subequations}\label{axLem02-02}
        \begin{equation}
            \begin{cases}
                \{ [a^n, \nabla a^n, b^n, \lambda^n, \omega^n, A^n] \}_{n = 1}^\infty \mbox{ is bounded in }
                \\
                \quad L^\infty(Q) \times L^\infty(Q)^N \times L^\infty(Q) \times L^\infty(Q) \times L^\infty(Q)^N \times L^\infty(Q)^{N \times N},
                \\[1ex]
                [a^n, \nabla a^n, b^n, \lambda^n, \omega^n, A^n] \to [a, \nabla a, b, \lambda, \omega, A] 
                \\
                \quad \mbox{in the pointwise sense a.e. in $ Q $, as $ n \to \infty $,}
            \end{cases}
            \label{axLem02-02_a}
        \end{equation}
        \begin{equation}\label{axLem02-02_b}
            \mbox{$ \{ \mu^n \}_{n = 1}^\infty $ is bounded in $ L^\infty(0, T; H) $, and $ \mu^n \to \mu $ in $ \scrH $, as $ n \to \infty $,}
        \end{equation}
\begin{equation}\label{axLem02-04}
\begin{cases}
\mbox{$ [p, \tilde{p}, z, \tilde{z}] \in \mathscr{X} $, $ \{ [p^n, \tilde{p}^n, z^n, \tilde{z}^n] \}_{n = 1}^\infty  \subset \mathscr{X} $,}
\\
[p^n, \tilde{p}^n, z^n, \tilde{z}^n] \to [p, \tilde{p}, z, \tilde{z}]  \mbox{ weakly in $ \mathscr{X} $,}
\\
\mbox{~~~~and } [p^n, z^n] \to [p, z] \mbox{ in $ [\scrH]^2 $, as $ n \to \infty $,}
\end{cases}
\end{equation}
and
\begin{equation}\label{axLem02-03-00}
    \begin{array}{c}
        \ds [h^n, k^n] := \mathcal{T}^n[p^n, \tilde{p}^n, z^n, \tilde{z}^n] \to [\tilde{h}, \tilde{k}] \mbox{ weakly in $ \mathscr{Y} $ as $ n \to \infty $, }
        \\[1ex]
        \mbox{for some $ [\tilde{h}, \tilde{k}] \in \mathscr{Y} $.}
    \end{array}
\end{equation}
    \end{subequations}
Then, it holds that:
\begin{equation}\label{axLem02-03}
    [\tilde{h}, \tilde{k}] = \mathcal{T}[p, \tilde{p}, z, \tilde{z}] = \mathcal{T}(a, b, \mu, \lambda, \omega, A)[p, \tilde{p}, z, \tilde{z}] \mbox{ in $ \mathscr{Y} $.}
\end{equation}
\end{keyLemma}
\begin{proof}
    Let us assume \eqref{axLem02-02}. Then, we may say that:
\begin{equation}\label{axLem02-10-01}
    \begin{cases}
        [p^n, z^n] \to [p, z] \mbox{ in the pointwise sense a.e. in $ Q $,}
        \\
        \mu^n \to \mu \mbox{ in the pointwise sense a.e. in $ Q $,}
        \\
        \mu^n(t) \to \mu(t) \mbox{ in $ H $, for a.e. $ t \in (0, T) $,}
    \end{cases} \mbox{as $ n \to \infty $,}
\end{equation}
by taking a subsequence if necessary. Also, by \eqref{axLem02-02_a}, \eqref{axLem02-04}, and the dominated convergence theorem (cf. \cite[Theorem 10]{MR0492147}), we arrive at:
\begin{subequations}\label{axLem02-10}
    \begin{equation}\label{axLem02-10_a}
        \lambda^n \varphi \to \lambda \varphi \mbox{ in $ \scrH $ and $ \varphi \omega^n \to \varphi \omega $, in $  [\scrH]^N $, for any $ \varphi \in \scrH $, as $ n \to \infty $,}
    \end{equation}
    \begin{equation}\label{axLem02-10_c}
        \begin{array}{c}
            \nabla (a^n \psi) = \psi \nabla a^n +a^n \nabla \psi \to \psi \nabla a +a \nabla \psi = \nabla (a \psi) \mbox{ in $ [\scrH]^N $},
            \\
            \mbox{i.e. } a^n \psi \to a \psi \mbox{ in $ \mathscr{V}_0 $, for any $ \psi \in \mathscr{V}_0 $, as $ n \to \infty $,}
        \end{array}
    \end{equation}
    and
    \begin{equation}\label{axLem02-10_b}
        \begin{array}{c}
            b^n \psi \to b \psi \mbox{ in $ \scrH $, } \omega^n \cdot \nabla \psi \to \omega \cdot \nabla \psi \mbox{ in $ \scrH $},
            \\
            \mbox{and } A^n \nabla \psi \to A \nabla \psi \mbox{ in $ [\scrH]^N $, for any $ \psi \in \mathscr{V}_0 $, as $ n \to \infty $.}
        \end{array}
    \end{equation}
\end{subequations}

Next, let us fix any test function $ [\varphi, \psi] \in \mathscr{V} \times \mathscr{V}_0 $. Then, as a consequence of \eqref{axLem02-02} and \eqref{axLem02-10}, it is observed that:
\begin{subequations}\label{axLem02-11}
\begin{equation}\label{axLem02-11_a}
\langle \partial_t \tilde{p}^n, \varphi \rangle_{\mathscr{V}} \to \langle \partial_t \tilde{p}, \varphi \rangle_{\mathscr{V}},
\end{equation}
\begin{equation}\label{axLem02-11_b}
    \langle Fp^n -p^n, \varphi \rangle_{\mathscr{V}} = ( \nabla p^n, \nabla \varphi )_{[\scrH]^N} \to ( \nabla p, \nabla \varphi )_{[\scrH]^N} = \langle Fp -p, \varphi \rangle_{\mathscr{V}},
\end{equation}
\begin{equation}\label{axLem02-11_c}
    \langle \lambda^n p^n, \varphi \rangle_{\mathscr{V}} = ( p^n, \lambda^n \varphi )_{\scrH} \to ( p, \lambda \varphi)_{\scrH} = \langle \lambda p, \varphi \rangle_{\mathscr{V}},
\end{equation}
\begin{equation}\label{axLem02_d}
    \langle \omega^n \cdot \nabla z^n, \varphi \rangle_{\mathscr{V}} = ( \nabla z^n, \varphi \omega^n )_{[\scrH]^N} \to ( \nabla z, \varphi \omega )_{[\scrH]^N} = \langle \omega \cdot \nabla z, \varphi \rangle_{\mathscr{V}}
\end{equation}
\begin{equation}\label{axLem02-11_e}
    \langle a^n \partial_t \tilde{z}^n, \psi \rangle_{\mathscr{V}_0} = \langle \partial_t \tilde{z}^n, a^n \psi \rangle_{\mathscr{V}_0} \to \langle \partial_t \tilde{z}, a \psi \rangle_{\mathscr{V}_0} = \langle a \partial_t \tilde{z}, \psi \rangle_{\mathscr{V}_0}, 
\end{equation}
\begin{equation}\label{axLem02-11_f}
\langle b^n z^n, \psi \rangle_{\mathscr{V}_0} = ( z^n, b^n \psi )_{\scrH} \to ( z, b \psi )_{\scrH} = \langle b z, \psi \rangle_{\mathscr{V}_0}, 
\end{equation}
\begin{align}\label{axLem02-11_g}
\langle -\mathrm{div} \, & ( A^n \nabla z^n +\nu \nabla z^n ), \psi \rangle_{\mathscr{V}_0} = ( \nabla z^n, A^n \nabla \psi +\nu \nabla \psi )_{[\scrH]^N}
\nonumber
\\
& \to  ( \nabla z, A \nabla \psi +\nu \nabla \psi )_{[\scrH]^N} = \langle -\mathrm{div} \, ( A \nabla z +\nu \nabla z ), \psi \rangle_{\mathscr{V}_0},
\end{align}
and
\begin{align}\label{axLem02-11_h}
\langle -\mathrm{div} \, & (p^n \omega^n), \psi \rangle_{\mathscr{V}_0} = ( p^n, \omega^n \cdot \nabla \psi )_{\scrH} 
\nonumber
\\
& \to ( p, \omega \cdot \nabla \psi )_{\scrH} = \langle -\mathrm{div} \, (p \omega), \psi \rangle_{\mathscr{V}_0}, \mbox{ as $ n\to \infty $.} 
\end{align}
\end{subequations}

Next by invoking \eqref{axLem02-02_b} and Remark \ref{Rem.sol} (\hyperlink{R1}{R1}), we compute that:
\begin{align}\label{axLem02-20}
|\sqrt{\mu^n}(t) & \varphi(t) -\sqrt{\mu}(t) \varphi(t)|_{H} \leq \left[ \int_\Omega |(\mu^n -\mu)(t)||\varphi(t)|^2 \, dx \right]^{\frac{1}{2}}
\nonumber
\\
    & \leq |(\mu^n -\mu)(t)|_H^{\frac{1}{2}} |\varphi(t)|_{L^4(\Omega)} \leq C_V^{L^4} |(\mu^n -\mu)(t)|_{H}^{\frac{1}{2}} |\varphi(t)|_{V} 
\nonumber
\\
& \to 0, \mbox{ as $ n \to \infty $, for a.e. $ t \in (0, T) $,}
\end{align}
and similarly, 
\begin{equation}\label{axLem02-21}
\begin{cases}
    \ds \bigl| (\sqrt{\mu^n} -\sqrt{\mu})(t) \varphi(t) \bigr|_H^2 \leq (C_V^{L^4})^2 \sup_{n \in \bbN} |\mu^n -\mu|_{L^\infty(0, T; H)} |\varphi(t)|_V^2, 
    \\
    \qquad \mbox{ a.e. $ t \in (0, T) $,}
\\[1ex]
\ds \sup_{n \in \bbN} \bigl| \sqrt{\mu^n} p^n \bigr|_{\scrH} \leq C_V^{L^4} \sup_{n \in \bbN} \bigl\{  |\mu^n|_{L^\infty(0, T; H)}^{\frac{1}{2}} |p^n|_{\mathscr{V}} \bigr\} < \infty.
\end{cases}
\end{equation}
Furthermore, by using \eqref{axLem02-02}, \eqref{axLem02-10-01}, \eqref{axLem02-20}, \eqref{axLem02-21}, Lions's lemma (cf. \cite[Lemma 1.3 on page 12]{MR0259693}), and the dominated convergence theorem (cf. \cite[Theorem 10]{MR0492147}), one can see that:
\begin{equation*}
\begin{cases}
\sqrt{\mu^n} \varphi \to \sqrt{\mu} \varphi \mbox{ in $ \scrH $,}
\\[0.5ex]
\sqrt{\mu^n} p^n \to \sqrt{\mu} p \mbox{ weakly in $ \scrH $, as $ n \to \infty $,}
\end{cases}
\end{equation*}
and therefore, 
\begin{align}\label{axLem02-23}
\langle \mu^n p^n, & \varphi \rangle_{\mathscr{V}} = ( \sqrt{\mu^n} p^n, \sqrt{\mu^n} \varphi )_{\scrH}
\nonumber
\\
& \to ( \sqrt{\mu} p, \sqrt{\mu} \varphi )_{\scrH} = \langle \mu p, \varphi \rangle_{\mathscr{V}} \mbox{ as $ n \to \infty $.}
\end{align}

Now, the conclusion \eqref{axLem02-03} is verified by taking into account \eqref{axLem02-03-00}, \eqref{axLem02-11}, and \eqref{axLem02-23}. 
\end{proof}

\begin{keyLemma}\label{axLem03}
    Let us take sextets $ [a^\ell, b^\ell, \mu^\ell, \lambda^\ell, \omega^\ell, A^\ell] \in \mathscr{S} $ and quartets $ [p^\ell, \tilde{p}^\ell, z^\ell, \tilde{z}^\ell] \in \mathscr{X} $, $ \ell = 1, 2 $, and let us set:
\begin{equation*}
\begin{array}{c}
    \ds \mathcal{T}^\ell := \mathcal{T} (a^\ell, b^\ell, \mu^\ell, \lambda^\ell, \omega^\ell, A^\ell) \mbox{ in } \mathscr{L}(\mathscr{X}; \mathscr{Y}), 
\\[1ex]
\mbox{and } [h^\ell, k^\ell] := \mathcal{T}^\ell [p^\ell, \tilde{p}^\ell, z^\ell, \tilde{z}^\ell] \mbox{ in $ \mathscr{Y} $, $ \ell = 1, 2 $.}
\end{array}
\end{equation*}
Besides, let $ \tilde{C}^* = \tilde{C}^*(a^1, b^1, \lambda^1, \omega^1) $ be a positive constant, depending on $ \delta_*(a^1) $, $ |b^1|_{L^\infty(Q)} $, $ |\lambda^1|_{L^\infty(Q)} $, and $ |\omega^1|_{L^\infty(Q)^N} $, which is defined as:
\begin{equation}\label{defOf_tildeC*}
\begin{array}{ll}
    \tilde{C}^* := & \ds \frac{9 (1 +\nu)}{\min \{1, \nu, \delta_*(a^1) \}} \bigl( 1 +(C_V^{L^4})^2 +(C_V^{L^4})^4 +(C_{V_0}^{L^4})^2 \bigr) \cdot
\\[2ex]
& \qquad \cdot \bigl( 1 +|b^1|_{L^\infty(Q)} +|\lambda^1|_{L^\infty(Q)} +|\omega^1|_{L^\infty(Q)^N}^2 \bigr).
\end{array}
\end{equation}
Then, it holds that:
\begin{align}\label{axLem03-01}
    & \bigl\langle \partial_t (\tilde{p}^1 -\tilde{p}^2)(t), (p^1 -p^2)(t) \bigr\rangle_{V} +\bigl\langle \partial_t (\tilde{z}^1 -\tilde{z}^2)(t), a^1(t) (z^1 -z^2)(t) \bigr\rangle_{V_0} 
\nonumber
\\
& ~~ +\frac{1}{2} |(p^1 -p^2)(t)|_{V}^2 +\frac{\nu}{2} |(z^1 -z^2)(t)|_{V_0}^2 
\nonumber
\\
\leq & \tilde{C}^* \left( 
    \ds |(p^1 -p^2)(t)|_H^2 +|\sqrt{\mbox{\small$ a^1(t) $}}(z^1 -z^2)(t)|_H^2  
\right) 
\\[1ex]
& ~~  +\tilde{C}^* \bigl( |(h^1 -h^2)(t)|_{V^*}^2 +|(k^1 -k^2)(t)|_{V_0^*}^2 +\tilde{R}^*(t) \bigr), 
\nonumber
\\[1ex]
& ~~ \mbox{for a.e. $ t \in (0, T) $,}
\nonumber
\end{align}
where 
\begin{align}\label{defOf_tildeR*}
    & \tilde{R}^*(t) := |\partial_t \tilde{z}^2(t)|_{V_0^*}^2 \bigl( |a^1 -a^2|_{L^\infty({Q})}^2 +|\nabla (a^1 -a^2)(t)|_{L^4(\Omega)^N}^2 \bigr)
\nonumber
\\[1ex]
& \quad +|p^2(t)|_{V}^2 \bigl( |(\mu^1 -\mu^2)(t)|_{H}^2 +|(\omega^1 -\omega^2)(t)|_{L^4(\Omega)^N}^2 \bigr) 
\nonumber
\\[1ex]
& \quad +|z^2(t)|_{V_0}^2 |(b^1 -b^2)(t)|_{L^4(\Omega)}^2 +|p^2(t) (\lambda^1 -\lambda^2)(t)|_H^2  
\nonumber
\\[1ex]
& \quad +|\nabla z^2(t) \cdot (\omega^1 -\omega^2)(t)|_H^2 +|(A^1 -A^2)(t) \nabla z^2(t)|_{[H]^N}^2,
    \\[1ex]
&  \mbox{ for  a.e. $ t \in (0, T) $.}
\nonumber
\end{align}
\end{keyLemma}
\begin{proof}
The conclusion \eqref{axLem03-01} will be deduced using the following identity:
\begin{align*}
\ds \bigl\langle \mathcal{T}^1[p^1, & \tilde{p}^1, z^1, \tilde{z}^1](t) -\mathcal{T}^2 [p^2, \tilde{p}^2, z^2, \tilde{z}^2](t), \bigl[ (p^1 -p^2)(t), (z^1 -z^2)(t) \bigr] \bigr\rangle_{V \times V_0}
\\
= & \bigl\langle \bigl[ (h^1 -h^2)(t), (k^1 -k^2)(t) \bigr], \bigl[ (p^1 -p^2)(t), (z^1 -z^2)(t) \bigr] \bigr\rangle_{V \times V_0}, 
\\
& \mbox{a.e. $ t \in (0, T) $.}
\end{align*}
Then, the constant  $ \tilde{C}^* = \tilde{C}^*(a^1, b^1, \lambda^1, \omega^1) > 0 $, as in \eqref{defOf_tildeC*}, and the function $ \tilde{R}^* : (0, T) \longrightarrow \bbR $, as in \eqref{defOf_tildeR*}, can be derived based on the following computations, after using the estimates, as in Remark \ref{Rem.sol} (\hyperlink{R1}{R1}), and H\"{o}lder's and Young's inequalities:
\begin{align*}
\bigl\langle (Fp^1 & -p^1)(t) -(Fp^2 -p^2)(t), (p^1 -p^2)(t) \bigr\rangle_{V} 
\nonumber
\\
= & |(p^1 -p^2)(t)|_V^2 -|(p^1 -p^2)(t)|_H^2, 
\mbox{ a.e. $ t \in (0, T) $;}
\end{align*}
\begin{align*}
    & \hspace{-6ex} \bigl\langle \bigl( \mu^1 p^1 -\mu^2 p^2 \bigr)(t), (p^1 -p^2)(t) \bigr\rangle_{V} = \bigl| \sqrt{\mu^1(t)}(p^1 -p^2)(t) \bigr|_{H}^2 +I_1(t),
\\[1ex]
    \mbox{with} \hspace{10ex} &
    \\
    I_1(t) := & \int_\Omega p^2(t) (\mu^1 -\mu^2)(t)(p^1 -p^2)(t) \, dx 
\\
\leq & (C_V^{L^4})^2 |(\mu^1 -\mu^2)(t)|_H |p^2(t)|_{V} |(p^1 -p^2)(t)|_{V}
\\
\leq & \frac{1}{4} |(p^1 -p^2)(t)|_V^2 +(C_V^{L^4})^4 |p^2(t)|_V^2 |(\mu^1 -\mu^2)(t)|_H^2, 
\\
& \qquad \mbox{a.e. $ t \in (0, T) $;} 
\end{align*}
\begin{align*}
    & \hspace{-6ex} \bigl\langle ( \lambda^1 p^1 -\lambda^2 p^2 )(t), (p^1 -p^2)(t) \bigr\rangle_{V} = I_2(t) +I_3(t),
\\[1ex]
    \mbox{with} \hspace{10ex} &
    \\
    I_2(t) := & \int_\Omega \lambda^1(t)|(p^1 -p^2)(t)|^2 \, dx \leq |\lambda^1|_{L^\infty(Q)} |(p^1 -p^2)(t)|_H^2,
\\[1ex]
    I_3(t) := & \int_\Omega p^2(t) (\lambda^1 -\lambda^2)(t) (p^1 -p^2)(t) \, dx 
\\
\leq & \frac{1}{2} |(p^1 -p^2)(t)|_H^2 +\frac{1}{2}\bigl| p^2(t)(\lambda^1 -\lambda^2)(t) \bigr|_{H}^2, \mbox{ a.e. $ t \in (0, T) $;} 
\end{align*}
\begin{align*}
    & \hspace{-6ex} \bigl\langle ( \omega^1 \cdot \nabla z^1 -\omega^2 \cdot \nabla z^2 )(t), (p^1 -p^2)(t) \bigr\rangle_{V} = I_4(t) +I_5(t), 
\\[1ex]
    \mbox{with} \hspace{10ex} &
    \\
    I_4(t) := & \int_\Omega \bigl( \omega^1 \cdot \nabla (z^1 -z^2) \bigr)(t) (p^1 -p^2)(t) \, dx 
\\
\leq & |\omega^1(t)|_{L^\infty(\Omega)^N} |\nabla (z^1 -z^2)(t)|_{[H]^N} |(p^1 -p^2)(t)|_H
\\
    \leq & \frac{\nu}{16} |(z^1 -z^2)(t)|_{V_0}^2 +\frac{4}{\nu} |\omega^1|_{L^\infty(Q)^N}^2 |(p^1 -p^2)(t)|_H^2,
\\[1ex]
I_5(t) := & \int_\Omega \bigl((\omega^1 -\omega^2) \cdot \nabla z^2 \bigr)(t) (p^1 -p^2)(t) \, dx
\\
\leq & \frac{1}{2} |(p^1 -p^2)(t)|_H^2 +\frac{1}{2}\bigl| \nabla z^2(t) \cdot (\omega^1 -\omega^2)(t) \bigr|_{H}^2, \mbox{ a.e. $ t \in (0, T) $;} 
\end{align*}
\begin{align*}
\bigl\langle (h^1 & -h^2)(t), (p^1 -p^2)(t) \bigr\rangle_{V} 
\\
& \leq \frac{1}{4} |(p^1 -p^2)(t)|_V^2 +|(h^1 -h^2)(t)|_{V^*}^2, \mbox{ a.e. $ t \in (0, T) $;}
\end{align*}
\begin{align*}
& \hspace{-6ex} \bigl\langle (a^1 \partial_t \tilde{z}^1)(t) -(a^2 \partial_t \tilde{z}^2)(t), (z^1 -z^2)(t) \bigr\rangle_{V_0}
\\
    = & \bigl\langle \partial_t (\tilde{z}^1 -\tilde{z}^2)(t), a^1(t) (z^1 -z^2)(t) \bigr\rangle_{V_0} +J_1(t)
\\
    \mbox{with} \hspace{10ex} &
    \\
    J_1(t) := & \bigl\langle \partial_t \tilde{z}^2(t), (a^1 -a^2)(t) (z^1 -z^2)(t) \bigr\rangle_{V_0} 
\\[1ex]
    \leq & |\partial_t \tilde{z}^2(t)|_{V_0^*} \bigl| \nabla \bigl( (a^1 -a^2)(z^1 -z^2)\bigr)(t) \bigr|_{[H]^N}
\\[1ex]
\leq & C_{V_0}^{L^4} |\partial_t \tilde{z}^2(t)|_{V_0^*} |\nabla (a^1 -a^2)(t)|_{L^4(\Omega)^N} |(z^1 -z^2)(t)|_{V_0}
\\[1ex]
& \quad +|\partial_t \tilde{z}^2(t)|_{V_0^*} |a^1 -a^2|_{L^\infty({Q})} |(z^1 -z^2)(t)|_{V_0}
\\[1ex]
    \leq & \frac{\nu}{8} |(z^1 -z^2)(t)|_{V_0}^2 +\frac{4}{\nu} |\partial_t \tilde{z}^2(t)|_{V_0^*}^2 |a^1 -a^2|_{L^\infty({Q})}^2
\\[1ex]
    & \quad +\frac{4 (C_{V_0}^{L^4})^2}{\nu} |\partial_t \tilde{z}^2(t)|_{V_0^*}^2 |\nabla (a^1 -a^2)(t)|_{L^4(\Omega)^N}^2, \mbox{ a.e. $ t \in (0, T) $;} 
\end{align*}
\begin{align*}
    & \hspace{-6ex} \bigl\langle b^1(t) z^1(t) -b^2(t) z^2(t), (z^1 -z^2)(t) \bigr\rangle_{V_0} = J_2(t) +J_3(t)
\\[1ex]
    \mbox{with} \hspace{10ex} &
    \\
    J_2(t) := & \int_\Omega b^1(t) |(z^1 -z^2)(t)|^2 \, dx  \leq |b^1|_{L^\infty(Q)} |(z^1 -z^2)(t)|_H^2
    \\
    \leq & \frac{|b^1|_{L^\infty(Q)}}{\delta_*(a^1)}|\sqrt{\mbox{\small$a^1(t)$}}(z^1 -z^2)(t)|_H^2, 
\end{align*}
\begin{align*}
    J_3(t) := & \int_\Omega z^2(t) (b^1 -b^2)(t) (z^1 -z^2)(t) \, dx 
    \\
    \leq & \frac{1}{\sqrt{\delta_*(a^1)}} |\sqrt{\mbox{\small$a^1(t)$}}(z^1 -z^2)(t)|_{H} \cdot C_{V_0}^{L^4}|z^2(t)|_{V_0} |(b^1 -b^2)(t)|_{L^4(\Omega)}
    \\
    \leq & \frac{1}{2 \delta_*(a^1)} |\sqrt{\mbox{\small$a^1(t)$}}(z^1 -z^2)(t)|_{H}^2 +\frac{(C_{V_0}^{L^4})^2}{2} |z^2(t)|_{V_0}^2|(b^1 -b^2)(t)|_{L^4(\Omega)}^2, 
    \\
    & \qquad \mbox{ a.e. $ t \in (0, T) $;} 
\end{align*}
\begin{align*}
    & \hspace{-6ex} -\bigl\langle \mathrm{div} \, \bigl( (A^1 \nabla z^1 +\nu \nabla z^1) -(A^2 \nabla z^2 +\nu \nabla z^2) \bigl) (t), (z^1 -z^2)(t) \bigr\rangle_{V_0}
    \\
    = & \bigl| [A^1]^{\frac{1}{2}}\nabla(z^1 -z^2)(t) \bigr|_{[H]^N}^2 +\nu |(z^1 -z^2)(t)|_{V_0}^2 +J_4(t) 
    \\[1ex]
    \mbox{with} \hspace{10ex} &
    \\
    J_4(t) := & \int_\Omega \bigl( (A^1 -A^2) \nabla z^2 \bigr) (t) \cdot \nabla (z^1 -z^2)(t) \, dx
    \\
    \leq & \frac{\nu}{16} |(z^1 -z^2)(t)|_{V_0}^2 +\frac{4}{\nu} |(A^1 -A^2)(t) \nabla z^2(t)|_{[H]^N}^2, \mbox{ a.e. $ t \in (0, T) $;} 
\end{align*}
\begin{align*}
    & \hspace{-6ex} -\bigl\langle \mathrm{div} \, \bigl( p^1 \omega^1 -p^2 \omega^2 \bigl) (t), (z^1 -z^2)(t) \bigr\rangle_{V_0} = J_5(t) +J_6(t),
    \\[1ex]
    \mbox{with} \hspace{10ex} &
    \\
    J_5(t) := & \int_\Omega (p^1 -p^2)(t) \bigl( \omega^1 \cdot \nabla (z^1 -z^2) \bigr)(t) \, dx
    \\
    \leq & \frac{\nu}{16} |(z^1 -z^2)(t)|_{V_0}^2 +\frac{4}{\nu}|\omega^1|_{L^\infty(Q)^N}^2 |(p^1 -p^2)(t)|_H^2,
    \\[1ex]
    J_6(t) := & \int_\Omega p^2(t) \bigl( (\omega^1 -\omega^2) \cdot \nabla (z^1 -z^2) \bigr)(t) \, dx
    \\
    \leq & C_{V}^{L^4} |\nabla (z^1 -z^2)(t)|_{[H]^N} |p^2(t)|_{V} |(\omega^1 -\omega^2)(t)|_{L^4(\Omega)^N}
    \\
    \leq & \frac{\nu}{16} |(z^1 -z^2)(t)|_{V_0}^2 +\frac{4 (C_{V}^{L^4})^2}{\nu} |p^2(t)|_{V}^2 |(\omega^1 -\omega^2)(t)|_{L^4(\Omega)^N}^2,
    \\
    & \qquad \mbox{ a.e. $ t \in (0, T) $;} 
\end{align*}
and
\begin{align*}
    \bigl\langle (k^1 & -k^2)(t), (z^1 -z^2)(t) \bigr\rangle_{V_0} 
    \\
    & \leq \frac{\nu}{8} |(z^1 -z^2)(t)|_{V_0}^2 +\frac{2}{\nu} |(k^1 -k^2)(t)|_{V_0^*}^2, \mbox{ a.e. $ t \in (0, T) $.}
\end{align*}
Thus, the proof is finished.
\end{proof}

\section{Proofs of Main Theorems}\label{04Proofs}

This Section is devoted to the proofs of Main Theorems \hyperlink{mainThrm01}{1} and \ref{mainThrm2}. The proofs are divided in four parts, listed below: 
\begin{description}
\item[\textrm{\textmd{$\S$\,\ref{subSec4.1}}}]Proof of Main Theorem \ref{mainThrm2} (\hyperlink{II-A}{II-A});
\item[\textrm{\textmd{$\S$\,\ref{subSec4.2}}}]Proof of existence in Main Theorem \hyperlink{mainThrm01}{1} (\hyperlink{I-A}{I-A});
\item[\textrm{\textmd{$\S$\,\ref{subSec4.3}}}]Proofs of continuous dependence in Main Theorem \hyperlink{mainThrm01}{1} (\hyperlink{I-B}{I-B}) and uniqueness in (\hyperlink{I-A}{I-A});
\item[\textrm{\textmd{$\S$\,\ref{subSec4.4}}}]Proof of Main Theorem \ref{mainThrm2} (\hyperlink{II-B}{II-B}).
\end{description}

\subsection{Proof of Main Theorem \ref{mainThrm2} (II-A)} \label{subSec4.1}

Let us set the constant $ \tau_0 = \tau_0(a, b, \lambda, \omega) \in (0, 1) $, given in \eqref{axLem01-03}, as the required constant $ \tau_* = \tau_*(a, b, \lambda, \omega) $ in this Section. Let us fix any time step-size $ \tau \in (0, \tau_*) $.  Since the value of constant $ \tau_* $ is independent of the time-index $ i \in \bbN $ and the choice of given data as in \eqref{axLem01-01}, the solution $ \{ [p_i, z_i]  \}_{i = 0}^\infty $ to the time-discrete scheme (\hyperlink{DP}{DP})$_\tau$ is obtained, by applying Key-Lemma \ref{axLem01} to the system \eqref{DP_1st} and \eqref{DP_2nd}, inductively, for every $ i \in \bbN $. \hfill $ \Box $ 

\subsection{Proof of existence in Main Theorem \ref{mainThrm1} (I-A)} \label{subSec4.2}

For simplicity of description, let us set:
\begin{equation*}
    \Delta_i^T := (t_{i -1}, t_i) \cap (0, T), \mbox{ for $ i = 1, 2, 3, \dots $,} 
\end{equation*}
and
\begin{equation*}
    \mathcal{T}_\tau := \mathcal{T}\bigl( [\overline{a}]_\tau, [\overline{b}]_\tau, [\overline{\mu}]_\tau, [\overline{\lambda}]_\tau, [\overline{\omega}]_\tau, [\overline{A}]_\tau \bigr) \mbox{ in $ \mathscr{L}(\mathscr{X}; \mathscr{Y}) $, for $ \tau \in (0, \tau_*) $.}
\end{equation*}
Then, invoking the definitions \eqref{DP_1st}, \eqref{DP_2nd}, and \eqref{op_A}, we can reformulate the time-discretization scheme (\hyperlink{DP}{DP})$_\tau$ as the following linear equation:
\begin{equation}\label{mThIA-00}
    \begin{cases}
        \mathcal{T}_\tau \bigl[[\overline{p}]_\tau, [{p}]_\tau, [\overline{z}]_\tau, [{z}]_\tau \bigr] = \bigl[ [\overline{h}]_\tau, [\overline{k}]_\tau \bigr] \mbox{ in } \mathscr{Y}, 
        \\[1ex]
        \bigl[ [p]_\tau(0), [z]_\tau(0) \bigr] = [p_0, z_0] \mbox{ in $ [H]^2 $,}
    \end{cases} \mbox{for $ \tau \in (0, \tau_*) $.}
\end{equation}

Next, for any $ \tau \in (0, \tau_*) $, let us denote by $ \tilde{C}_\tau^* $ the constant $ \tilde{C}^* $ as in \eqref{defOf_tildeC*}, corresponding to the case $ [a^1, b^1, \lambda^1, \omega^1] = \bigl[ [\overline{a}]_\tau, [\overline{b}]_\tau, [\overline{\lambda}]_\tau, [\overline{\omega}]_\tau \bigr] $. Also, let $ C_0^* $ be the constant $ C^* $ as in \eqref{defOf_C*} corresponding to the case $ [a^1, b^1, \lambda^1, \omega^1] = [a, b,  \lambda, \omega] $. Notice that under the setting of \eqref{D-given_a}, \eqref{D-given_b}, \eqref{defOf_C*}, and \eqref{defOf_tildeC*}, the constant $ C_0^* $ is a uniform upper-bound for the sequence $ \{ \tilde{C}_\tau^* \}_{\tau \in (0, 1)} $, more precisely:
\begin{align}
    C_0^* -\tilde{C}_\tau^* & \,  \geq \frac{9 (1 +\nu)}{\min \{1, \nu, \delta_*(a) \}} |a|_{W^{1, \infty}(Q)} \geq 0, \mbox{ for any $ \tau \in (0, \tau_*) $.}
    \label{mThIIB-02-01}
\end{align}

Next, let us apply Key-Lemma \ref{axLem03} to the following case:
\begin{equation*}
\begin{cases}
    \mathcal{T}^{\ell} = \mathcal{T}_\tau \mbox{ in $ \mathscr{L}(\mathscr{X}; \mathscr{Y}) $, $ \ell = 1, 2 $,}
\\[1ex]
[p^1, \tilde{p}^1, z^1, \tilde{z}^1] = \bigl[[\overline{p}]_\tau, [{p}]_\tau, [\overline{z}]_\tau, [{z}]_\tau \bigr]
\\
\qquad \mbox{and } [p^2, \tilde{p}^2, z^2, \tilde{z}^2] = [0, 0, 0, 0] \mbox{ in $ \mathscr{X} $,} 
\\[1ex]
[h^1, k^1] = \bigl[ [\overline{h}]_\tau, [\overline{k}]_\tau \bigr] \mbox{ and } [h^2, k^2] = [0, 0] \mbox{ in $ \mathscr{Y} $,}
\end{cases} \mbox{ for $ \tau \in (0, \tau_*) $.}
\end{equation*}
Then, from \eqref{axLem03-01} and \eqref{mThIIB-02-01}, we deduce that:
\begin{align} \label{mThIA-01}
\bigl< \partial_t [p]_\tau & (t), [\overline{p}]_\tau(t) \bigr>_{V} +\bigl< \partial_t [z]_\tau(t), [\overline{a}]_\tau(t) [\overline{z}]_\tau(t) \bigr>_{V_0} 
\nonumber
\\
& +\frac{1}{2} |[\overline{p}]_\tau(t)|_{V}^2 +\frac{\nu}{2} |[\overline{z}]_\tau(t)|_{V_0}^2
\nonumber
\\
    \leq & C_0^* \bigl( |[\overline{p}]_\tau(t)|_{H}^2 +|\sqrt{[\overline{a}]_\tau(t)}[\overline{z}]_\tau(t)|_{H}^2 \bigr)
\nonumber
\\
    & +C_0^* \bigl( |[\overline{h}]_\tau(t)|_{V^*}^2 +|[\overline{k}]_\tau(t)|_{V_0^*}^2 \bigr), \mbox{ for a.e. $ t \in (0, T) $.}
\end{align}
Also, since:
\begin{align*}
\bigl< \partial_t [p]_\tau & (t), [\overline{p}]_\tau(t) \bigr>_{V} +\bigl< \partial_t [z]_\tau(t), [\overline{a}]_\tau(t) [\overline{z}]_\tau(t) \bigr>_{V_0}
\\
\geq & \frac{1}{2 \tau} \int_\Omega \bigl( |p_i|^2 -|p_{i -1}|^2 \bigr) \, dx +\frac{1}{2 \tau} \int_\Omega \bigl( a_{i} |z_i|^2 -a_{i -1} |z_{i -1}|^2 \bigr) \, dx
\\
& \quad -\frac{1}{2 \tau} \int_\Omega |a_{i} -a_{i -1}| |z_{i -1}|^2 \, dx,
\\
\geq & \frac{1}{2 \tau} \bigl( |p_i|_H^2 -|p_{i -1}|_H^2 \bigr) +\frac{1}{2 \tau} \bigl( |\sqrt{a_{i}} z_i|_H^2 -|\sqrt{a_{i -1}} z_{i -1}|_H^2 \bigr)
\\
& \quad -\frac{|\partial_t [a]_\tau|_{L^\infty(Q)}}{2 \delta_*(a)} |\sqrt{a_{i -1}} z_{i -1}|_H^2
\\
    \geq & \frac{1}{2 \tau} \bigl( |p_i|_H^2 -|p_{i -1}|_H^2 \bigr) +\frac{1}{2 \tau} (1 +C_0^* \tau) \bigl( |\sqrt{a_{i}} z_i|_H^2 -|\sqrt{a_{i -1}} z_{i -1}|_H^2 \bigr)
\\
    & \quad -\frac{C_0^*}{2} |\sqrt{a_{i}} z_i|_H^2, \mbox{ for any $ t \in \Delta_i^T $, $ i = 1, 2, 3, \dots $, } 
\end{align*}
the inequality \eqref{mThIA-01} can be reduced to:
\begin{align} \label{mThIA-02}
    \bigl( |p_i|_H^2 & -|p_{i -1}|_H^2 \bigr) +(1 +C_0^* \tau)\bigl( |\sqrt{a_{i}} z_i|_H^2 -|\sqrt{a_{i -1}} z_{i -1}|_H^2 \bigr) 
\nonumber
\\[1ex]
& +\tau \bigl( |p_i|_{V}^2 +\nu |z_i|_{V_0}^2 \bigr) 
\nonumber
\\[1ex]
    \leq 3 C_0^*  \tau & \bigl( |p_i|_{H}^2 +|\sqrt{a_{i}} z_i|_{H}^2 \bigr) +2 C_0^* \tau \bigl( |h_i|_{V^*}^2 +|k_i|_{V_0^*}^2 \bigr), 
\\[1ex]
    & \mbox{for any $ \tau \in (0, \tau_*) $, and $ i = 1, 2, 3, \dots $.}
\nonumber
\end{align}
Here, let us take $ \delta_0 \in (0, \tau_*) $ so small to satisfy:
\begin{equation} \label{tau0}
    \delta_0 < \frac{1}{6 C_0^*}, \mbox{ and in particular, } \frac{1}{1 -3 C_0^* \delta_0} < 2,
\end{equation}
and having in mind \eqref{timeInterp}, \eqref{D-given}, \eqref{mThIIB-02-01}, and \eqref{tau0}, let us apply the discrete version of Gronwall's lemma \cite[Section 3.1]{emmrich1999discrete} to \eqref{mThIA-02}. Then, it is observed that:
\begin{align} \label{mThIA-03}
    & |[\overline{p}]_\tau(t)|_H^2 +\delta_*(a) |[\overline{z}]_\tau(t)|_H^2 +\int_0^t \bigl( |[\overline{p}]_\tau(\varsigma)|_{V}^2 +\nu |[\overline{z}]_\tau(\varsigma)|_{V_0}^2 \bigr) \, d \varsigma 
\nonumber
\\
    \leq & |p_i|_H^2 +(1 +C_0^* \tau)|\sqrt{a_{i}} z_i|_H^2 +\tau \sum_{j = 1}^{i} \bigl( |p_j|_{V}^2 +\nu |z_j|_{V_0}^2  \bigr)
\nonumber
\\
    \leq & e^{\frac{3 C_0^* t_i}{1 -3 C_0^* \tau}} \left[ \bigl( |p_0|_H^2 +(1 +C_0^* \tau)|\sqrt{a_{0}} z_0|_H^2 \bigr) +2 C_0^* \tau \sum_{j = 1}^{i} \bigl( |h_j|_{V^*}^2 +|k_j|_{V_0^*}^2 \bigr) \right]
\nonumber
\\
    \leq & 2 (1 +C_0^* +|a|_{L^\infty(Q)}) e^{6 C_0^* T +1} \left[ \bigl( |p_0|_H^2 +|z_0|_H^2 \bigr) +\bigl| \bigl[ [\overline{h}]_\tau, [\overline{k}]_\tau \bigr] \bigr|_{\mathscr{Y}}^2 \right],
\nonumber
\\
    & \mbox{for all $ \tau \in (0, \delta_0) $, $ t \in \Delta_i^T $, and $ i = 1, 2, 3, \dots $.}
\end{align}
Hence, putting:
\begin{equation*}
    C_1^* := \left( \frac{2 \bigl( 1 +C_0^* +|a|_{L^\infty(Q)} \bigr) e^{6 C_0^* T +1}}{\min \{ 1, \nu, \delta_*(a) \}} \right)^{\frac{1}{2}},
\end{equation*}
it is deduced from \eqref{rxForces01} and \eqref{mThIA-03} that:
\begin{equation}\label{mThIA-10}
\hspace{-2ex}\begin{cases}
    \ds \sup_{\tau \in (0, \tau_0)} \left\{ \bigl| \bigl[ [\overline{p}]_\tau, [\overline{z}]_\tau \bigr] \bigr|_{L^\infty(0, T; H)^2}, \bigl| \bigl[ [\underline{p}]_\tau, [\underline{z}]_\tau \bigr] \bigr|_{L^\infty(0, T; H)^2}, \bigl| \bigl[ [{p}]_\tau, [{z}]_\tau \bigr] \bigr|_{C([0, T]; H)^2} \right\}
\\[1ex]
    \qquad \leq C_1^* \bigl( |[p_0, z_0]|_{[H]^2} +\bigl| \bigl[ [\overline{h}]_\tau, [\overline{k}]_\tau \bigr] \bigr|_{\mathscr{Y}} \bigr) \leq C_1^*  \bigl( |[p_0, z_0]|_{[H]^2} +K^* \bigr),
\\[1ex]
    \ds \sup_{\tau \in (0, \tau_0)} \bigl| \bigl[ [\overline{p}]_\tau, [\overline{z}]_\tau \bigr] \bigr|_{\mathscr{V} \times \mathscr{V}_0}  \leq C_1^*  \bigl( |[p_0, z_0]|_{[H]^2} +\bigl| \bigl[ [\overline{h}]_\tau, [\overline{k}]_\tau \bigr] \bigr|_{\mathscr{Y}} \bigr)
\\[1ex]
    \ds\qquad \leq C_1^*  \bigl( |[p_0, z_0]|_{[H]^2} +K^* \bigr). 
\end{cases}
\end{equation}
Meanwhile, in the light of the estimate \eqref{estFinal01} as in Remark \ref{Rem.4concl1}, it is inferred that: 
\begin{align*}
    \bigl| \bigl[ \partial_t [p]_\tau, \partial_t [z]_\tau & \bigr] \bigr|_{\mathscr{Y}} \leq M_{1, \tau} \bigl( \bigl| \bigl[ [\overline{h}]_\tau, [\overline{k}]_\tau \bigr] \bigr|_{\mathscr{Y}} +\bigl| \bigl[ [\overline{p}]_\tau, [\overline{z}]_\tau \bigr] \bigr|_{\mathscr{V} \times \mathscr{V}_0} \bigr)
\nonumber
\\
\leq ~ & M_{1, \tau} \left( C_1^* \bigl( |[p_0, z_0]|_{[H]^2} +K^* \bigr) +K^* \right), \mbox{ for all $ \tau \in (0, \delta_0) $,} 
\end{align*}
where for any $ \tau \in (0, \delta_0) $, $ M_{1, \tau} $ is the constant given in \eqref{defOf_M1} corresponding to the case $ [a, b, \mu, \lambda, \omega, A] = \big[ [\overline{a}]_\tau, [\overline{b}]_\tau, [\overline{\mu}]_\tau, [\overline{\lambda}]_\tau, [\overline{\omega}]_\tau, [\overline{A}]_\tau \bigr] $. Additionally, invoking the settings \eqref{D-given} and \eqref{defOf_M1}, we estimate that:
    \begin{align*}
        M_{1, \tau} := & M_1\bigl(  [\overline{a}]_\tau, [\overline{b}]_\tau, [\overline{\mu}]_\tau, [\overline{\lambda}]_\tau, [\overline{\omega}]_\tau, [\overline{A}]_\tau \bigr) \leq M_1(a, b, \mu, \lambda, \omega, A) =: M_1,
    \end{align*}
    and
\begin{align}\label{mThIA-13}
    \bigl| \bigl[ \partial_t [p]_\tau, \partial_t [z]_\tau & \bigr] \bigr|_{\mathscr{Y}} \leq M_1 (1 +C_1^* +K^*)^2 \bigl( 1 +|[p_0, z_0]|_{[H]^2} \bigr), \mbox{ for all $ \tau \in (0, \delta_0) $.} 
\end{align}

Now, referring to the general theories of compactness (cf. \cite{MR0916688,MR1375237}), e.g. theories of Aubin's type \cite[Corollary 4]{MR0916688}, Arzer\`{a}--Ascoli \cite[Theorem 1.3.1]{MR1375237}, Alaoglu--Bourbaki--Kakutani \cite[Theorem 1.2.5]{MR1375237}, and so on, it is ensured that the estimates \eqref{mThIA-10} and \eqref{mThIA-13} imply the existence of a pair of functions $ [p, z] \in [\scrH]^2 $, together with a sequence:
\begin{subequations}\label{mThIA-20}
\begin{equation}\label{mThIA-20_a}
\delta_0 > \tau_1 > \tau_2 > \cdots > \tau_n \downarrow 0 \mbox{ as $ n \to \infty $;}
\end{equation}
fulfilling:
    \begin{align}\label{mThIA-20_b}
    \bigl[ [\overline{p}]_{\tau_n}, [\overline{z} & ]_{\tau_n} \bigr] \to [p, z] \mbox{ weakly in $ \mathscr{V} \times \mathscr{V}_0 $,}
    \nonumber
    \\
    & \mbox{and weakly-$*$ in $ L^\infty(0, T; H)^2 $,}
\end{align}
\begin{align}\label{mThIA-20_c}
\bigl[ [{p}]_{\tau_n}, & [{z}]_{\tau_n} \bigr] \to [p, z] \mbox{ in $ C([0, T]; V^*) \times C([0, T]; V_0^*) $,}
\nonumber
\\
& \mbox{in $ L^2(\delta, T; H)^2 $, for any $ \delta \in (0, T) $,} 
\nonumber
\\
& \mbox{weakly in $ W^{1, 2}(0, T; V^*) \times W^{1, 2}(0, T; V_0^*) $,} 
\nonumber
\\
    & \mbox{and weakly-$*$ in $ L^\infty(0, T; H)^2 $, as $ n \to \infty $,} 
\end{align}
and in particular, 
    \begin{equation}\label{mThIA-20_d}
        [p(0), z(0)] = \bigl[ [p]_{\tau_n}(0), [z]_{\tau_n}(0) \bigr] = [p_0, z_0] \mbox{ in $ [H]^2 $, for $ n = 1,2, 3, \dots $.}
    \end{equation}
    Here, from \eqref{mThIA-10} and \eqref{mThIA-20_c}, it follows that:
    \begin{align*}
        & ~ \varlimsup_{n \to \infty} \bigl| \bigl[ [p]_{\tau_n}, [z]_{\tau_n} \bigr] -[p, z] \bigr|_{[\scrH]^2}^2 
        \\
        \leq & ~ \delta \cdot \sup_{n \in \bbN} \bigl| \bigl[ [p]_{\tau_n}, [z]_{\tau_n} \bigr] -[p, z] \bigr|_{L^\infty(0, T; H)^2}^2 +\varlimsup_{n \to \infty} \bigl| \bigl[ [p]_{\tau_n}, [z]_{\tau_n} \bigr] -[p, z] \bigr|_{L^2(\delta, T; H)^2}^2 
        \\
        \leq & ~ \delta \cdot 4(C_1^*)^2 \bigl( |[p_0, z_0]|_{[H]^2} +K^* \bigr)^2, \mbox{ for any $ \delta \in (0, T) $.}
    \end{align*}
    Thus, 
    \begin{align}\label{mThIA-20_e}
        \bigl[ [{p}]_{\tau_n}, & [{z}]_{\tau_n} \bigr] \to [p, z] \mbox{ in $ [\scrH]^2 $, as $ n \to \infty $,}
    \end{align}
    and moreover, the estimate \eqref{mThIA-10} allows us to derive:
    \begin{align}\label{mThIA-20_f}
        \bigl[ [\overline{p}]_{\tau_n}, & [\overline{z}]_{\tau_n} \bigr] \to [p, z] \mbox{ and } \bigl[ [\underline{p}]_{\tau_n}, [\underline{z}]_{\tau_n} \bigr] \to [p, z], \mbox{ in $ [\scrH]^2 $, as $ n \to \infty $.}
    \end{align}
\end{subequations}

\noindent
On account of \eqref{D-given}, \eqref{mThIA-00}, \eqref{mThIA-20}, (\hyperlink{A1}{A1}), Remark \ref{Rem.t-discrete}, and the dominated convergence theorem (cf. \cite[Theorem 10]{MR0492147}), it is checked that the following sequences:
\begin{equation*}
    \begin{cases}
        \{ [a^n, b^n, \mu^n, \lambda^n, \omega^n, A^n] \}_{n = 1}^\infty := \bigl\{ \bigl[ [\overline{a}]_{\tau_n}, [\overline{b}]_{\tau_n}, [\overline{\mu}]_{\tau_n}, [\overline{\lambda}]_{\tau_n}, [\overline{\omega}]_{\tau_n}, [\overline{A}]_{\tau_n} \bigr] \bigr\}_{n = 1}^\infty, 
        \\ 
        \{ [p^n, \tilde{p}^n, z^n, \tilde{z}^n] \}_{n = 1}^\infty := \bigl\{ \bigl[ [\overline{p}]_{\tau_n}, [p]_{\tau_n}, [\overline{z}]_{\tau_n}, [z]_{\tau_n} \bigr] \bigr\}_{n = 1}^\infty,
        \\
        \{ [h^n, k^n] \}_{n = 1}^\infty := \bigl\{ \mathcal{T}_{\tau_n} [p^n, \tilde{p}^n, z^n, \tilde{z}^n] = \mathcal{T} (a^n, b^n, \mu^n, \lambda^n, \omega^n, A^n) [p^n, \tilde{p}^n, z^n, \tilde{z}^n] \bigr\}_{n = 1}^\infty,
    \end{cases} 
\end{equation*}
enjoy the assumptions \eqref{axLem02-02} of Key-Lemma \ref{axLem02}, in the case when $ [\tilde{p}, \tilde{z}] = [p, z] $ and $ [\tilde{h}, \tilde{k}] = [h, k] $. Therefore, as a consequence of the Key-Lemma \ref{axLem02}, we can show that:
\begin{equation}\label{mThIA-100}
    [h, k] = \mathcal{T} (a, b, \mu, \lambda, \omega, A)[p, p, z, z] \mbox{ in $ \mathscr{Y} $.} 
\end{equation}
This implies that $ [p, z] $ is the solution to (\hyperlink{P}{P}), with the initial condition verified in \eqref{mThIA-20_d}. \hfill \qed

\subsection{\hspace{-1ex}Proofs of continuous dependence in Main Theorem \ref{mainThrm1} (I-B) and uniqueness in (I-A)} \label{subSec4.3}

We begin by confirming the following relationships:
\begin{align*}
    \mathcal{T}^\ell[p^\ell, p^\ell, z^\ell, z^\ell] := & ~ \mathcal{T}(a^\ell, b^\ell, \mu^\ell, \lambda^\ell, \omega^\ell, A^\ell)[p^\ell, p^\ell, z^\ell, z^\ell]
\\
= & ~ [h^\ell, k^\ell] \mbox{ in $ \mathscr{Y} $, $ \ell = 1, 2 $,}
\end{align*}
for the two solutions $ [p^\ell, z^\ell] $ to (\hyperlink{P}{P}), $ \ell = 1, 2 $. On this basis, let us apply Key-Lemma \ref{axLem03} to the case when $ [\tilde{p}^\ell, \tilde{z}^\ell] = [p^\ell, z^\ell] $, $ \ell = 1, 2 $. Then, with \eqref{defOf_C*}, \eqref{defOf_R*}, and \eqref{defOf_tildeC*}--\eqref{defOf_tildeR*} in mind, we infer that:
\begin{align}\label{mThIA-21}
    \frac{1}{2} \frac{d}{dt} & \bigl( |(p^1 -p^2)(t)|_H^2 +\bigl| \sqrt{\mbox{\small$ a^1(t) $}} (z^1 -z^2)(t) \bigr|_H^2 \bigr)
\nonumber
\\
& \quad +\frac{1}{2} \bigl( |(p^1 -p^2)(t)|_V^2 +\nu |(z^1 -z^2)(t)|_{V_0}^2 \bigr)
\nonumber
\\
    \leq & \frac{3C^*}{2} \bigl( |(p^1 -p^2)(t)|_H^2 +\bigl| \sqrt{\mbox{\small$ a^1(t) $}} (z^1 -z^2)(t) \bigr|_H^2 \bigr)
\nonumber
    \\
    & \quad + C^* \bigl( |(h^1 -h^2)(t)|_{V^*}^2 +|(k^1 -k^2)(t)|_{V_0^*}^2 +R^*(t) \bigr), 
\\ & \mbox{a.e. } t \in (0, T); \nonumber
\end{align}
via the following computations:
\begin{equation*}
    C^* -\tilde{C}^* \geq \frac{9 (1 +\nu)}{\min \{1, \nu, \delta_*(a^1) \}} |a^1|_{W^{1, \infty}(Q)} \geq 0,  ~ R^* = \tilde{R}^* \mbox{ a.e. in $ (0, T) $,}
\end{equation*}
and
\begin{align*}
    & \hspace{-6ex} \bigl< \partial_t (z^1 -z^2)(t), a^1(t) (z^1 -z^2)(t) \bigr>_{V_0} = \frac{1}{2} \frac{d}{dt} \bigl| \sqrt{\mbox{\small$ a^1(t) $}}(z^1 -z^2)(t) \bigr|_H^2 -J_7(t),
    \\
    \mbox{with} \hspace{10ex} & 
    \\
    J_7(t) :=  & \frac{1}{2} \int_\Omega \partial_t a^1(t) |(z^1 -z^2)(t)|^2 \, dx \leq \frac{|\partial_t a^1|_{L^\infty(Q)}}{2} \int_\Omega |(z^1 -z^2)(t)|^2 \, dx
\\
\leq & \frac{|\partial_t a^1|_{L^\infty(Q)}}{2 \delta_*(a^1)} \bigl| \sqrt{\mbox{\small$ a^1(t) $}} (z^1 -z^2)(t) \bigr|_{H}^2
\\
    \leq & \frac{C^*}{2} \bigl| \sqrt{\mbox{\small$ a^1(t) $}} (z^1 -z^2)(t) \bigr|_{H}^2, \mbox{ a.e. $ t \in (0, T) $,}
\end{align*}
where we have also used Remark~\ref{Rem.relat_P-S}.
\medskip

Now, the conclusion \eqref{contiDep} of (\hyperlink{I-B}{I-B}) will be obtained by applying Gronwall's lemma to \eqref{mThIA-21}. 
\medskip

Furthermore, we can conclude the uniqueness part of (\hyperlink{I-A}{I-A}), immediately, by using the inequality \eqref{contiDep} in the special case when $ [a^1, b^1, \mu^1, \lambda^1, \omega^1, A^1] = [a^2, b^2, \mu^2, \lambda^2, \omega^2, A^2] $, $ [p_0^1, z_0^1] = [p_0^2, z_0^2] $, and $ [h^1, k^1] = [h^2, k^2] $.  \hfill \qed 

\subsection{Proof of Main Theorem \ref{mainThrm2} (II-B)} \label{subSec4.4}
    On account of \eqref{mThIA-10}, \eqref{mThIA-20_b}, \eqref{mThIA-20_c}, \eqref{mThIA-20_e}, \eqref{mThIA-20_f}, \eqref{mThIA-100}, and the uniqueness result as in Main Theorem \hyperlink{mainThrm01}{1} (\hyperlink{I-A}{I-A}), we can easily verify the convergence \eqref{concluConv_a}, together with:
\begin{equation}\label{mThIIB-20}
    \begin{cases}
        [[\overline{p}]_{\tau}, [\overline{z}]_{\tau}] \to [p, z] \mbox{ in $ [\scrH]^2 $, and weakly in $ \mathscr{V} \times \mathscr{V}_0 $,}
        \\
        [[\underline{p}]_{\tau}, [\underline{z}]_{\tau}] \to [p, z] \mbox{ in $ [\scrH]^2 $,}
    \end{cases}
    \mbox{as $ \tau \downarrow 0 $.}
\end{equation}

Now, let us assume $ [p_0, z_0] \in V \times V_0 $, and verify the convergence \eqref{concluConv_c}. Let us apply Key-Lemma \ref{axLem03} to the case when:
\begin{subequations}\label{mThIIB-01}
    \begin{equation}\label{mThIIB-01_a}
        \begin{cases}
            [a^1, b^1, \mu^1, \lambda^1, \omega^1, A^1] = \bigl[ [\overline{a}]_\tau, [\overline{b}]_\tau, [\overline{\mu}]_\tau, [\overline{\lambda}]_\tau, [\overline{\omega}]_\tau, [\overline{A}]_\tau \bigr],
            \\
            [a^2, b^2, \mu^2, \lambda^2, \omega^2, A^2] = [a, b, \mu, \lambda, \omega, A],
        \end{cases}
    \end{equation}
    \begin{equation}\label{mThIIB-01_b}
        \begin{cases}
            [p^1, \tilde{p}^1, z^1, \tilde{z}^1] = \bigl[ [\overline{p}]_\tau, [p]_\tau, [\overline{z}]_\tau, [z]_\tau \bigr], 
            \\
            [p^2, \tilde{p}^2, z^2, \tilde{z}^2] = [p, p, z, z],
        \end{cases}
    \end{equation}
    and
    \begin{equation*}
        [h^1, k^1] = \bigl[ [\overline{h}]_\tau, [\overline{k}]_\tau \bigr] \mbox{ and } [h^2, k^2] = [h, k].
    \end{equation*}
\end{subequations}
Then, we can see that:
\begin{align}\label{mThIIB-02}
    \bigl< \partial_t ([p]_\tau & -p)(t), ([\overline{p}]_\tau -p)(t) \bigr>_{V} +\bigl< \partial_t ([z]_\tau -z)(t), [\overline{a}]_\tau(t) ([\overline{z}]_\tau -z)(t) \bigr>_{V_0}
    \nonumber
    \\
    & +\frac{1}{2} \bigl| ([\overline{p}]_\tau -p)(t) \bigr|_{V}^2 +\frac{\nu}{2} \bigl| ([\overline{z}]_\tau -z)(t) \bigr|_{V_0}^2
    \nonumber
    \\
    \leq & \tilde{C}_\tau^* \bigl( \bigl| ([\overline{p}]_\tau -p)(t) \bigr|_H^2 +\bigl| \sqrt{[\overline{a}]_\tau(t)} ([\overline{z}]_\tau -z)(t) \bigr|_H^2 \bigr)
    \nonumber
    \\
    & +\tilde{C}_\tau^* \bigl( |([\overline{h}]_\tau -h)(t)|_{V^*}^2 +|([\overline{k}]_\tau -k)(t)|_{V_0}^2 +R_\tau^*(t) \bigr),
    \\
    & \mbox{for a.e. $ t \in (0, T) $, and any $ \tau \in (0, \delta_0) $.}
    \nonumber
\end{align}
Here, for any $ \tau \in (0, \delta_0) $, $ \tilde{C}_\tau^* $ is the constant as in \eqref{mThIIB-02-01}, and $ R_\tau^* $ is the function of time as in \eqref{defOf_tildeR*}, corresponding to the case of \eqref{mThIIB-01}. 
\medskip

Next, let us invoke \eqref{timeInterp}, and compute that:
\begin{subequations}\label{mThIIB-03}
    \begin{align}\label{mThIIB-03_a}
        & \hspace{-6ex} \bigl< \partial_t ([p]_\tau -p)(t), ([\overline{p}]_\tau -p)(t) \bigr>_{V} 
        \nonumber
        \\
        & \qquad = \frac{1}{2} \frac{d}{dt} \bigl| ([p]_\tau -p)(t) \bigr|_H^2 +\frac{1}{\tau} I_0^*(t) +I_1^*(t),
        \\
        \mbox{with} \hspace{10ex} &
        \nonumber
        \\
        I_0^*(t) := & \tau \bigl< \partial_t [p]_\tau(t), ([\overline{p}]_\tau -[p]_\tau)(t) \bigr>_{V}
        \nonumber
        \\
        = & \sum_{i = 1}^\infty \left( \mbox{\small$ \ds \frac{t_i -t}{\tau} $} \right) \chi_{\Delta_i^T}(t) \bigl| ([\overline{p}]_\tau -[\underline{p}]_\tau)(t) \bigr|_H^2, 
        \label{mThIIB-03_b}
        \\
        I_1^*(t) := & -\bigl< \partial_t p(t), ([\overline{p}]_\tau -[p]_\tau)(t) \bigr>_V, \mbox{ a.e. $ t \in (0, T) $,}
        \label{mThIIB-03_c}
    \end{align}
\end{subequations}
    \begin{align}
        \label{mThIIB-04}
        \bigl| ([\overline{p}]_\tau & -p)(t) \bigr|_H^2 \leq 2 \bigl( \bigl| ([{p}]_\tau -p)(t) \bigr|_H^2 + \bigl| ([\overline{p}]_\tau -[p]_\tau)(t) \bigr|_H^2 \bigr)
        \nonumber
        \\
        = & 2 \bigl| ([{p}]_\tau -p)(t) \bigr|_H^2 +2  \sum_{i = 1}^\infty \left( \mbox{\small$ \ds \frac{t_i -t}{\tau} $} \right)^2 \chi_{\Delta_i^T}(t) \bigl| ([\overline{p}]_\tau -[\underline{p}]_\tau)(t) \bigr|_H^2
        \\
        \leq &  2 \bigl| ([{p}]_\tau -p)(t) \bigr|_H^2 +2 I_0^*(t), \mbox{ a.e. $ t \in (0, T) $,}
        \nonumber
    \end{align}
\begin{subequations}\label{mThIIB-05}
    \begin{align}
        & \hspace{-6ex} \bigl< \partial_t ([z]_\tau -z)(t), [\overline{a}]_\tau(t) ([\overline{z}]_\tau -z)(t) \bigr>_{V_0} 
        \nonumber
        \\
        = & \frac{1}{2} \frac{d}{dt} \bigl| \sqrt{a(t)} ([z]_\tau -z)(t) \bigr|_H^2 +\frac{1}{\tau} J_0^*(t) +J_1^*(t) +J_2^*(t) +J_3^*(t),
        \label{mThIIB-05_a}
        \\
        \mbox{with} \hspace{10ex} &  
        \nonumber
        \\[2ex]
        J_0^*(t) := ~ & \tau \bigl< \partial_t [z]_\tau(t), a(t)([\overline{z}]_\tau -[z]_\tau)(t) \bigr>_{V_0}
        \label{mThIIB-05_b}
        \\
        = & \sum_{i = 1}^\infty \left( \mbox{\small$ \ds \frac{t_i -t}{\tau} $} \right) \chi_{\Delta_i^T}(t) \bigl| \sqrt{a(t)} ([\overline{z}]_\tau -[\underline{z}]_\tau)(t) \bigr|_H^2, 
        \nonumber
        \end{align}
        \begin{align}
        J_1^*(t) := & -\frac{1}{2} \int_\Omega \partial_t a(t) |([z]_\tau -z)(t)|^2 \, dx
        \nonumber
        \\
        \geq & -\frac{|\partial_t a|_{L^\infty(Q)}}{\delta_*(a)} \bigl| \sqrt{a(t)}([z]_\tau -z)(t) \bigr|_H^2,
        \nonumber
        \\
        \geq & -C_0^* \bigl| \sqrt{a(t)}([z]_\tau -z)(t) \bigr|_H^2,
        \label{mThIIB-05_c}
        \\[2ex]
        J_2^*(t) := & \bigl< \partial_t ([z]_\tau -z)(t), ([\overline{a}]_\tau -a)(t) ([\overline{z}]_\tau -z)(t) \bigr>_{V_0}
        \nonumber
        \\
            \geq & -|\partial_t([z]_\tau -z)(t)|_{V_0^*} \bigl| ([\overline{a}]_\tau -a)(t) \nabla([\overline{z}]_\tau -z)(t) \bigr|_{[H]^N}
        \nonumber
        \\
            & \quad -|\partial_t([z]_\tau -z)(t)|_{V_0^*} \bigl| ([\bar{z}_\tau -z](t)) \nabla ([\bar{a}]_\tau -a) \bigr|_{[H]^N},
        \nonumber
        \\
            \geq & -|[\overline{a}]_\tau -a|_{L^\infty(Q)} |\partial_t([z]_\tau -z)(t)|_{V_0^*} |([\overline{z}]_\tau -z)(t)|_{V_0}
        \nonumber
        \\
        & \quad -2|a|_{W^{1, \infty}(Q)}|\partial_t([z]_\tau -z)(t)|_{V_0^*} |([\bar{z}]_\tau -z)(t)|_{H},
        \label{mThIIB-05_d}
        \\[2ex]
            J_3^*(t) := & -\bigl\langle \partial_t z(t), a(t)([\overline{z}]_\tau -[z]_\tau)(t) \bigr\rangle_{V_0},
            \\
            & \mbox{ a.e. $ t \in (0, T) $,}
            \nonumber
    \end{align}
\end{subequations}
\begin{align}\label{mThIIB-08}
    \bigl| \sqrt{[\overline{a}]_\tau(t)} & ([\overline{z}]_\tau -z)(t) \bigr|_H^2 \leq \frac{|[\overline{a}]_\tau|_{L^\infty(Q)}}{\delta_*(a)} \bigl| \sqrt{a(t)}([\overline{z}]_\tau -z)(t) \bigr|_H^2 
    \nonumber
    \\
    \nonumber
    \leq & \frac{2 |a|_{L^\infty(Q)}}{\delta_*(a)} \bigl( \bigl| \sqrt{a(t)}([{z}]_\tau -z)(t) \bigr|_H^2 +\bigl| \sqrt{a(t)}([\overline{z}]_\tau -[{z}]_\tau)(t) \bigr|_H^2 \bigr)
    \nonumber
    \\
    \leq & C_0^* \bigl| \sqrt{a(t)} ([{z}]_\tau -z)(t) \bigr|_H^2 
    \nonumber
    \\
    & + C_0^* \sum_{i = 1}^\infty  \left( \mbox{\small$ \ds \frac{t_i -t}{\tau} $} \right)^2\chi_{\Delta_i^T}(t)  \bigl| \sqrt{a(t)} ([\overline{z}]_\tau -[\underline{z}]_\tau)(t) \bigr|_H^2
    \nonumber
    \\
    \leq &  C_0^* \bigl| \sqrt{a(t)} ([{z}]_\tau -z)(t) \bigr|_H^2 +C_0^* J_0^*(t), \mbox{ a.e. $ t \in (0, T) $.}
\end{align}
Here, let us set a constant $ \delta_1 \in (0, \delta_0) $ of time step-size, so small to satisfy:
\begin{equation}\label{mThIIB-06}
    0 < \delta_1 < \frac{1}{2} \min \left\{ \delta_0, \frac{1}{(1 +C_0^*)^2} \right\}.
\end{equation}
Then, from \eqref{mThIIB-02}--\eqref{mThIIB-06}, we can deduce that:
\begin{align}\label{mThIIB-07}
    \frac{1}{2} \frac{d}{dt} \bigl( |(& [p]_\tau -p)(t)|_H^2 +\bigl| \sqrt{a(t)}([z]_\tau -z)(t) \bigr|_H^2 \bigr)
    \nonumber
    \\
    & +\frac{1}{2} |([\overline{p}]_\tau -p)(t)|_V^2 +\frac{\nu}{2} |([\overline{z}]_\tau -z)(t)|_{V_0}^2
    \nonumber
    \\
    \leq & (1 +C_0^*)^2 \bigl( |([p]_\tau -p)(t)|_H^2 +\bigl| \sqrt{a(t)}([z]_\tau -z)(t) \bigr|_H^2 \bigr) -J_1^*(t)
    \nonumber
    \\
    & +C_0^* \bigl( |([\overline{h}]_\tau -h)(t)|_{V^*}^2 +|([\overline{k}]_\tau -k)(t)|_{V_0^*}^2 +R_\tau^*(t) \bigr)
    \nonumber
    \\
    & -\frac{1}{\tau} \left( 1 -\tau (1 +C_0^*)^2 \right) (I_0^* +J_0^*)(t) -I_1^*(t) -J_2^*(t) -J_3^*(t)
    \nonumber
    \\
    \leq & 2 (1 +C_0^*)^2 \bigl( |([p]_\tau -p)(t)|_H^2 +\bigl| \sqrt{a(t)}([z]_\tau -z)(t) \bigr|_H^2 \bigr)
\nonumber
    \\
    & +C_0^* \bigl( |([\overline{h}]_\tau -h)(t)|_{V^*}^2 +|([\overline{k}]_\tau -k)(t)|_{V_0^*}^2 +R_\tau^*(t) \bigr) 
    \\
    & +|I_1^*(t)| +|J_2^*(t)| +|J_3^*(t)|, \mbox{a.e. $ t \in (0, T) $, and any $ \tau \in (0, \delta_1) $.}
    \nonumber
\end{align}
    Additionally, taking into account \eqref{timeInterp}, \eqref{mThIA-10}, and $ [p_0, z_0] \in V \times V_0 $, we can observe that:
    \begin{equation*}
        \left\{ \begin{array}{rcl}
            \ds \sup_{\tau \in (0, \delta_1)} \bigl| \bigl[ [\underline{p}]_\tau, [\underline{z}]_\tau \bigr] \bigr|_{\mathscr{V} \times \mathscr{V}_0} & \leq & \ds  \sqrt{\delta_1} |[p_0, z_0]|_{V \times V_0} +C_1^* \bigl( |[p_0, z_0]|_{[H]^2} +K^* \bigr), 
            \\[1ex]
            \ds \sup_{\tau \in (0, \delta_1)} \bigl| \bigl[ [p]_\tau, [z]_\tau \bigr] \bigr|_{\mathscr{V} \times \mathscr{V}_0} & \leq & \ds  \sqrt{\delta_1} |[p_0, z_0]|_{V \times V_0} +2C_1^* \bigl( |[p_0, z_0]|_{[H]^2} +K^* \bigr), 
        \end{array} \right.
    \end{equation*}
and therefore, owing to \eqref{mThIIB-20}, and the compactness theory of Alaoglu--Bourbaki--Kakutani \cite[Theorem 1.2.5]{MR1375237}, 
\begin{align}\label{mThIIB-20_b}
    \bigl[ [\underline{p}]_\tau, [\underline{z}]_\tau \bigr] \to [p, z] \mbox{ and } \bigl[ [p]_\tau, [z]_\tau \bigr] \to [p, z], \mbox{ weakly in $ \mathscr{V} \times \mathscr{V}_0 $, as $ \tau \downarrow 0 $.} 
\end{align}
In the light of \eqref{D-given}, \eqref{defOf_tildeR*}, \eqref{mThIA-10}, \eqref{mThIA-13}, \eqref{mThIIB-20}, \eqref{mThIIB-20_b}, Remark \ref{Rem.t-discrete}, Remark \ref{Rem.sol} (\hyperlink{R2}{R2}), and the dominated convergence theorem (cf. \cite[Theorem 10]{MR0492147}), it is verified that:
\begin{align}\label{mThIIB-09}
    |R_\tau^* & |_{L^1(0, T)} \leq |\partial_t z|_{\mathscr{V}_0^*}^2 |[\overline{a}]_\tau -a|_{L^\infty(Q)}^2 
    \nonumber
    \\
    & +\int_0^T |z(t)|_{V_0}^2 \bigl| ([\overline{b}]_\tau -b)(t) \bigr|_{L^4(\Omega)}^2 \, dt
    \nonumber
    \\
    & +\int_0^T |\partial_t z(t)|_{V_0^*}^2 |\nabla ([\overline{a}]_\tau -a)(t)|_{L^4(\Omega)^N}^2 \, dt
    \nonumber
    \\
    & +\int_0^T |p(t)|_V^2 \bigl( | ([\overline{\mu}]_\tau -\mu)(t)|_H^2 +|([\overline{\omega}]_\tau -\omega)(t)|_{L^4(\Omega)^N}^2 \bigr) \, dt
    \nonumber
    \\
    & +\int_0^T \bigl( \bigl| p(t)([\overline{\lambda}]_\tau -\lambda)(t) \bigr|_{H}^2 +\bigl| \nabla z(t) \cdot ([\overline{\omega}]_\tau -\omega)(t) \bigr|_H^2 \bigr) \, dt
    \nonumber
    \\
    & +\int_0^T \bigl| ([\overline{A}]_\tau -A)(t) \nabla z(t) \bigr|_{[H]^N}^2 \, dt \to 0,
\end{align}
\begin{equation}\label{mThIIB-10-00}
    |I_1^*|_{L^1(0, T)} \leq \bigl| \bigl< \partial_t p, [\overline{p}]_\tau -p \bigr>_{\mathscr{V}} \bigr| +\bigl| \bigl< \partial_t p, [p]_\tau -p \bigr>_{\mathscr{V}} \bigr| \to 0,
\end{equation}
\begin{align}\label{mThIIB-10}
    |J_2^*|_{L^1(0, T)} \leq & |\partial_t([z]_\tau -z)|_{\mathscr{V}_0^*} |[\overline{z}]_\tau -z|_{\mathscr{V}_0} |[\overline{a}]_\tau -a|_{L^\infty(Q)} 
    \nonumber
    \\
    & \quad +2|a|_{W^{1, \infty}(Q)} |\partial_t([z]_\tau -z)|_{\mathscr{V}_0^*}  |[\overline{z}]_\tau -z|_{\scrH} \to  0,
\end{align}
and
\begin{align}\label{mThIIB-10-01}
    |J_3^*|_{L^1(0, T)} \leq & |\langle a \partial_t z, [\overline{z}]_\tau -z \rangle_{\mathscr{V}_0}| +|\langle a \partial_t z, [z]_\tau -z \rangle_{\mathscr{V}_0}| 
    \nonumber
    \\
    \to & 0, \mbox{ as $ \tau \downarrow 0 $.}
\end{align}

Now, on account of \eqref{rxForces01}, \eqref{mThIIB-07}, \eqref{mThIIB-09}--\eqref{mThIIB-10-01}, we obtain the following convergence:
\begin{equation}\label{mThIIB-11}
    \begin{cases}
        \bigl[ [p]_\tau, [z]_\tau \bigr] \to [p, z] \mbox{ in $ C([0, T]; H)^2 $,}
        \\
        \bigl[ [\overline{p}]_\tau, [\overline{z}]_\tau \bigr] \to [p, z] \mbox{ in $ \mathscr{V} \times \mathscr{V}_0 $,}
    \end{cases}
    \mbox{as $ \tau \downarrow 0 $,}
\end{equation}
by applying Gronwall's lemma to \eqref{mThIIB-07}, and by letting $ \tau \downarrow 0 $ for the resulted inequality. 
\medskip

Moreover, with \eqref{timeInterp}, \eqref{mThIIB-20_b}, \eqref{mThIIB-11}, and $ [p_0, z_0] \in V \times V_0 $ in mind, we can see that:
\begin{align}\label{mThIIB-12}
    \bigl| [p, z] \bigr|_{\mathscr{V} \times \mathscr{V}_0}^2 \leq & \varliminf_{\tau \downarrow 0} \bigl| \bigl[ [\underline{p}]_\tau, [\underline{z}]_\tau \bigr] \bigr|_{\mathscr{V} \times \mathscr{V}_0}^2 \leq \varlimsup_{\tau \downarrow 0} \bigl| \bigl[ [\underline{p}]_\tau, [\underline{z}]_\tau \bigr] \bigr|_{\mathscr{V} \times \mathscr{V}_0}^2 
    \nonumber
    \\
    \leq & \lim_{\tau \downarrow 0} \left( \tau \bigl| [p_0, z_0] \bigr|_{V \times V_0}^2 +\bigl| \bigl[ [\overline{p}]_\tau, [\overline{z}]_\tau \bigr] \bigr|_{\mathscr{V} \times \mathscr{V}_0}^2 \right) = \bigl| [p, z] \bigr|_{\mathscr{V} \times \mathscr{V}_0}^2,
    \nonumber
    \\
    \mbox{i.e., } & \lim_{\tau \downarrow 0}  \bigl| \bigl[ [\underline{p}]_\tau, [\underline{z}]_\tau \bigr] \bigr|_{\mathscr{V} \times \mathscr{V}_0} = \bigl| [p, z] \bigr|_{\mathscr{V} \times \mathscr{V}_0}.
\end{align}
Subsequently, from \eqref{mThIIB-20_b}, \eqref{mThIIB-12} and the uniform convexity of the topology of $ \mathscr{V} \times \mathscr{V}_0 $, it follows that:
\begin{equation}\label{mThIIB-13}
    \bigl[ [\underline{p}]_\tau, [\underline{z}]_\tau \bigr] \to [p, z] \mbox{ in $ \mathscr{V} \times \mathscr{V}_0 $, as $ \tau \downarrow 0 $.}
\end{equation}
By \eqref{timeInterp}, \eqref{mThIIB-11}, and \eqref{mThIIB-13}, we infer that:
\begin{align}\label{mThIIB-14}
    \bigl| \bigl[ [{p}]_\tau, [{z}]_\tau \bigr]  -[p, z] \bigl|_{\mathscr{V} \times \mathscr{V}_0} 
    \nonumber
    & \leq \bigl| \bigl[ [\overline{p}]_\tau, [\overline{z}]_\tau \bigr]  -\bigl[ [\underline{p}]_\tau, [\underline{z}]_\tau \bigr] \bigl|_{\mathscr{V} \times \mathscr{V}_0} 
    \\
    & +\bigl| \bigl[ [\overline{p}]_\tau, [\overline{z}]_\tau \bigr]  -[p, z] \bigl|_{\mathscr{V} \times \mathscr{V}_0}
    \to  0, \mbox{ as $ \tau \downarrow 0 $.}
\end{align}

The convergence \eqref{concluConv_c} will be deduced as a consequence of \eqref{mThIIB-11} and \eqref{mThIIB-14}. \hfill \qed

\section{Conclusions}\label{05Concls}

In this Section, several Corollaries and Remarks will be shown as conclusions derived from Main Theorems and Key-Lemmas as in Sections \ref{02MainThms} and \ref{03Key-Lems}. 
\medskip

We begin by prescribing additional notations. Let us define:
\begin{equation*}
    \mathscr{Z} := \bigl( W^{1, 2}(0, T; V^*) \cap \mathscr{V} \bigr) \times  \bigl( W^{1, 2}(0, T; V_0^*) \cap \mathscr{V}_0 \bigr).
\end{equation*}
Then, for any sextet of data $ [a, b, \mu, \lambda, \omega, A] \in \mathscr{S}_0 $, the Main Theorem \hyperlink{mainThrm01}{1} allows us to define an operator $ \mathcal{P} = \mathcal{P}(a, b, \mu, \lambda, \omega, A) : [H]^2 \times \mathscr{Y} \longrightarrow \mathscr{Z} $, which maps any pair $ \bigl[ [p_0, z_0], [h, k] \bigr] $ of the initial pair $ [p_0, z_0] \in [H]^2 $ and the forcing pair $ [h, k] \in \mathscr{Y} $ to the corresponding solution $ [p, z] $ to (\hyperlink{P}{P}), i.e.:
\begin{equation*}
    [p, z] = \mathcal{P}\bigl[[p_0, z_0], [h, k] \bigr] = \mathcal{P}(a, b, \mu, \lambda, \omega, A)\bigl[[p_0, z_0], [h, k] \bigr] \in \mathscr{Z}.
\end{equation*}
\begin{remark}\label{Rem.2concl1}
    Let us fix any sextet of data $ [a, b, \mu, \lambda, \omega, A] \in \mathscr{S}_0 $. Then, from Remark \ref{Rem.4concl1}, the operator $ \mathcal{P} : [H]^2 \times \mathscr{Y} \longrightarrow \mathscr{Z} $ coincides with the inverse of:
    \begin{align}\label{op.R}
        \mathcal{Q} = & \mathcal{Q}(a, b, \mu, \lambda, \omega, A): [p, z] \in \mathscr{Z} 
        \nonumber
        \\
        \mapsto & \bigl[ [p_0, z_0], [h, k] \bigr] := \bigl[ [p(0), z(0)], \mathcal{T}(a, b, \mu, \lambda, \omega, A)[p, p, z, z] \bigr] \in [H]^2 \times \mathscr{Y}.
    \end{align}
    Namely, $ \mathcal{P} = \mathcal{P}(a, b, \mu, \lambda, \omega, A) $ $ \bigl( = \bigl[ \mathcal{Q}(a, b, \mu, \lambda, \omega, A) \bigr]^{-1} \bigr) $ is a bijective operator from $ [H]^2 \times \mathscr{Y} $ onto $ \mathscr{Z} $. Hence, the Hilbert space $ \mathscr{Z} $ coincides with the class $ \mathcal{P}([H]^2 \times \mathscr{Y}) $ of solutions to (\hyperlink{P}{P}), for all initial pairs $ [p_0, z_0] \in [H]^2 $ and forcing pairs $ [h, k] \in \mathscr{Y} $.  
\end{remark}
\begin{remark}\label{Rem3.concl1}
    Notice that $ \mathscr{Z} $ is a Hilbert space, endowed with the following inner product:
    \begin{align}
        & \bigl( [p, z], [\tilde{p}, \tilde{z}] \bigr)_{\mathscr{Z}} := (p, \tilde{p})_{W^{1, 2}(0, T; V^*)} +(p, \tilde{p})_{\mathscr{V}}
        \nonumber
        \\
        & \qquad +(z, \tilde{z})_{W^{1, 2}(0, T; V_0^*)} +(z, \tilde{z})_{\mathscr{V}_0}, \mbox{ for $ [p, z], [\tilde{p}, \tilde{z}] \in \mathscr{Z} $,}
        \label{Z_Hilbert01}
    \end{align}
and the norm:
    \begin{equation}
        \bigl| \, [p, z] \, \bigr|_\mathscr{Z} := \bigl[ \bigl( [p, z], [p, z] \bigr)_\mathscr{{Z}} \bigr]^{\frac{1}{2}}, \mbox{ for $ [p, z] \in \mathscr{Z} $.}
        \label{Z_Hilbert02}
    \end{equation}

    In addition, the Hilbert space $ \mathscr{Z} $ is a subspace of Banach spaces $ C([0, T]; H)^2 $ and $ L^\infty(0, T; H)^2 $, and $ \mathscr{Z} $ is compactly embedded in $ L^\infty(0, T; H)^2 $ in the weak-$ * $ topology. In fact, if $ \mathscr{B}_0 $ is a bounded subset in $ \mathscr{Z} $, then by referring to the compactness theories of Aubin's type \cite[Corollary 4]{MR0916688} and Alaoglu--Bourbaki--Kakutani \cite[Theorem 1.2.5]{MR1375237}, we find a sequence $ \{ [\hat{p}^n, \hat{z}^n] \}_{n = 1}^\infty \subset \mathscr{B}_0 $ such that:
    \begin{equation*}
        \begin{array}{c}
            [\hat{p}^n, \hat{z}^n] \to [\hat{p}, \hat{z}] \mbox{ in $ [\scrH]^2 $, and weakly in $ \mathscr{Z} $,}
            \\
            \mbox{as $ n \to \infty $, for some $ [\hat{p}, \hat{z}] \in \mathscr{Z} $.}
        \end{array}
    \end{equation*}
    Subsequently, from the strong convergence in $ [\scrH]^2 $, one can see that:
    \begin{equation*}
        [\hat{p}^n(t_0), \hat{z}^n(t_0)] \to [\hat{p}(t_0), \hat{z}(t_0)] \mbox{ in $ [H]^2 $, as $ n \to \infty $, for some $ t_0 \in (0, T) $,}
    \end{equation*}
    by taking a subsequence if necessary. Hence, we can observe the weak-$*$ compactness of $ \mathscr{B}_0 $ in $ L^\infty(0, T; H)^2 $, via the following estimate for the sequence $ \{ [\hat{p}^n, \hat{z}^n] \}_{n = 1}^\infty \subset \mathscr{B}_0 $:
    \begin{align*}
        & \frac{1}{2} \bigl| [\hat{p}^n(t), \hat{z}^n(t)] \bigr|_{[H]^2}^2 \leq \frac{1}{2} \bigl| [\hat{p}^n(t_0), \hat{z}^n(t_0)] \bigr|_{[H]^2}^2 
        \\
        & \qquad +\left| \int_{t_0}^t \bigl< \partial_t [\hat{p}^n(\varsigma), \hat{z}^n(\varsigma)], [\hat{p}^n(\varsigma), \hat{z}^n(\varsigma)] \bigr>_{V \times V_0} \, d \varsigma \right|
        \\
        \leq & \sup_{n \in \bbN} \left\{ \bigl| [\hat{p}^n(t_0), \hat{z}^n(t_0)] \bigr|_{[H]^2}^2 +\bigl| [\partial_t \hat{p}^n, \partial_t \hat{z}^n] \bigr|_{\mathscr{Y}} \bigl| [\hat{p}^n, \hat{z}^n] \bigr|_{\mathscr{V} \times \mathscr{V}_0} \right\} < \infty,
        \\
        & \qquad \mbox{for all $ t \in [0, T] $, and $ n \in \bbN $.}
    \end{align*}
    
    Furthermore, by the inclusion $ \mathscr{Z} \subset C([0, T]; H)^2 $, we can let $ \mathscr{Z} $ be a Banach space by using another norm:
    \begin{equation}\label{z_Banach01}
        \bigl\| \, [p, z] \, \bigr\| := \bigl| \, [p, z] \, \bigr|_\mathscr{Z} +\bigl| \, [p, z] \, \bigr|_{C([0, T]; H)^2}, \mbox{ for $ [p, z] \in \mathscr{Z} $.}
    \end{equation}
    In view of this, we denote by $ (\mathscr{Z}, \|\cdot\|) $ the Banach space with the norm as in \eqref{z_Banach01}, and distinguish it from the Hilbert space $ \mathscr{Z} $ with the usual topology as in \eqref{Z_Hilbert01} and \eqref{Z_Hilbert02}. 
\end{remark}

Based on these, we can derive the following Corollaries. 
\begin{corollary}\label{Cor.concl1}
    Let us fix a sextet of data $ [a, b, \mu, \lambda, \omega, A] \in \mathscr{S}_0 $. Then, the following two items hold.
    \begin{description}
        \item[\textmd{\hypertarget{C0}{\textit{(C0)}}}]
            There is a positive constant $ M_0^* = M_0^*(a, b, \mu, \lambda, \omega, A) $, depending on $ |a|_{C(\overline{Q})} $, \linebreak  $ |\nabla a|_{L^\infty(Q)^N} $, $ |b|_{L^\infty(Q)} $, $ |\mu|_{L^\infty(0, T; H)} $, $ |\lambda|_{L^\infty(Q)} $, $ |\omega|_{L^\infty(Q)^N} $, and $ |A|_{L^\infty(Q)^{N \times N}} $, and there is another positive constant $ M_1^* = M_1^*(a, b, \mu, \lambda, \omega, A) $, depending on $ |a|_{C(\overline{Q})} $, \linebreak $ |\nabla a|_{L^\infty(Q)^N} $, $ \delta_*(a) $, $ |b|_{L^\infty(Q)} $, $ |\mu|_{L^\infty(0, T; H)} $, $ |\lambda|_{L^\infty(Q)} $, $ |\omega|_{L^\infty(Q)^N} $, and $ |A|_{L^\infty(Q)^{N \times N}} $, such that:
                \begin{equation}\label{estFinal100}
                    \begin{array}{c}
                        M_0^* \bigl| \bigl[ [p_0, z_0], [h, k] \bigr] \bigr|_{[H]^2 \times \mathscr{Y}} \leq \bigl| [p, z] \bigr|_{(\mathscr{Z}; \|\cdot\|)} \leq M_1^*  \bigl| \bigl[ [p_0, z_0], [h, k] \bigr] \bigr|_{[H]^2 \times \mathscr{Y}},
                        \\[1ex]
                        \mbox{for all $ [p_0, z_0] \in [H]^2 $, $ [h, k] \in \mathscr{Y} $,}
                        \\[1ex]
                        \mbox{and $ [p, z] = \mathcal{P}(a, b, \mu, \lambda, \omega, A) \bigl[ [p_0, z_0], [h, k] \bigr] \in \mathscr{Z} $,}
                    \end{array}
                \end{equation}
                i.e. the operator $ \mathcal{P} = \mathcal{P}(a, b, \mu, \lambda, \omega, A) $ is an isomorphism between the Hilbert space $ [H]^2 \times \mathscr{Y} $ and the Banach space $ (\mathscr{Z}, \|\cdot\|) $.

        \item[\textmd{\hypertarget{C1}{\textit{(C1)}}}]The class of solutions $ \mathcal{P}([H]^2 \times \mathscr{Y}) $ is decomposed as follows:
            \begin{equation}\label{concl1-01}
                \mathcal{P}([H]^2 \times \mathscr{Y}) = \mathcal{P}\bigl( [H]^2 \times \{ [0, 0] \} \bigr) +\mathcal{P}\bigl( \{ [0, 0] \} \times \mathscr{Y} \bigr),  
            \end{equation}
            and moreover, the restricted class $ \mathcal{P}\bigl( [V \times V_0] \times \mathscr{Y} \bigr) $ is represented as: 
            \begin{equation}\label{concl1-02}
                \mathcal{P}\bigl( [V \times V_0] \times \mathscr{Y} \bigr) = [V \times V_0] +\mathcal{P} \bigl( \{ [0, 0] \} \times \mathscr{Y} \bigr), 
            \end{equation}
            by identifying $ V \times V_0 $ with a class of constants in time.
    \end{description}
\end{corollary}
\begin{proof}
    First, we show (\hyperlink{C0}{C0}). On account of Remarks \ref{Rem.4concl1} and \ref{Rem.2concl1}, we can see that the operator $ \mathcal{Q} = \mathcal{Q}(a, b, \mu, \lambda, \omega, A) : \mathscr{Z} \longrightarrow [H]^2 \times \mathscr{Y} $ given in \eqref{op.R} is a bijective linear operator. Besides, we can check that $ \mathcal{Q} = \mathcal{Q}(a, b, \mu, \lambda, \omega, A) \in \mathscr{L}\bigl( (\mathscr{Z}, \|\cdot\|); [H]^2 \times \mathscr{Y} \bigr) $, by using 
        the estimate \eqref{estFinal00}, and the trivial estimate:
    \begin{equation}\label{est.trivial}
        \bigl| [p(0), z(0)] \bigr|_{[H]^2} \leq M_0 \bigl| [p, z] \bigr|_{C([0, T]; H)^2}, \mbox{ for $ [p, z] \in \mathscr{Z} $,}
    \end{equation}
    with use of the constant $ M_0 = M_0(a, b, \mu, \lambda, \omega, A) $ given in \eqref{defOf_M0}. The estimates \eqref{estFinal00} and \eqref{est.trivial} imply that:
    \begin{align}\label{estFinal10}
        & \bigl| \mathcal{Q}(a, b, \mu, \lambda, \omega, A) \bigr|_{\mbox{\scriptsize$ \mathscr{L}\bigl( (\mathscr{Z}, \|\cdot\|);[H]^2 \times \mathscr{Y} \bigr) $}} \leq M_0, 
        \nonumber
        \\
        \mbox{i.e. } & \bigl| \mathcal{P}(a, b, \mu, \lambda, \omega, A) \bigr|_{\mbox{\scriptsize$ \mathscr{L}\bigl([H]^2 \times \mathscr{Y}; (\mathscr{Z}, \|\cdot\|) \bigr) $}} \geq \frac{1}{M_0} =: M_0^*, 
        \\
        & \mbox{if $  \mathcal{P}(a, b, \mu, \lambda, \omega, A) $ is bounded.} 
        \nonumber
    \end{align}

    Now, our task is reduced to verify the boundedness of linear operator $ \mathcal{P} = $ \linebreak $ \mathcal{P}(a, b, \mu, \lambda, \omega, A) : [H]^2 \times \mathscr{Y} \longrightarrow (\mathscr{Z}, \|\cdot\|) $. To this end, we apply Main Theorem \hyperlink{mainThrm01}{1} (\hyperlink{I-B}{I-B}) to the case when:
    \begin{equation*}
        [a^\ell, b^\ell, \mu^\ell, \lambda^\ell, \omega^\ell, A^\ell] = [a, b, \mu, \lambda, \omega, A], ~ \ell = 1, 2,
    \end{equation*}
    \begin{equation*}
        [p^1, z^1] = [p, z], ~ [p^2, z^2] = [0, 0],
    \end{equation*}
    \begin{equation*}
        \mbox{$ [h^1, k^1] $ is taken as an arbitrary $ [h, k] \in \mathscr{Y} $, and } [h^2, k^2] = [0, 0].
    \end{equation*}
    Then, we have:
    \begin{align*}
        \min \{ 1, & \nu, \delta_*(a) \} \bigl( \bigl| [p, z] \bigr|_{C([0, T]; H)^2}^2 +\bigl| [p, z] \bigr|_{\mathscr{V} \times \mathscr{V}_0}^2 \bigr)
        \nonumber
        \\
        \leq & 4 e^{3C_0^* T} \max \{ 1, |a|_{C(\overline{Q})} \} \bigl| [p_0, z_0] \bigr|_{[H]^2}^2 +8 C_0^* e^{3 C_0^* T} \bigl| [h, k] \bigr|_{\mathscr{Y}}^2,
    \end{align*}
    so that
    \begin{align}\label{concl1-10}
        \bigl| [p, z] \bigr|_{C([0, T]; H)^2}^2 & +\bigl| [p, z] \bigr|_{\mathscr{V} \times \mathscr{V}_0}^2 
        \leq (C_0^*)^2 e^{3 C_0^* T} \bigl( \bigl| [p_0, z_0] \bigr|_{[H]^2}^2 +\bigl| [h, k] \bigr|_{\mathscr{Y}}^2 \bigr),
    \end{align}
    where $ C_0^* = C_0^*(a, b, \lambda, \omega) $ is the constant as in \eqref{mThIIB-02-01}, i.e. the constant $ C^* $ as in \eqref{defOf_C*} corresponding to the case $ [a^1, b^1, \lambda^1, \omega^1] = [a, b, \lambda, \omega] $. Also, from the estimate \eqref{estFinal01} as in Remark \ref{Rem.4concl1}, and \eqref{concl1-10}, one can see that:
    \begin{equation}\label{concl1-11}
        \begin{cases}
            \bigl| [p, z] \bigr|_{\mathscr{Y}} \leq \sqrt{T} (C_{H}^{V^*} +C_{H}^{V_0^*})  \bigl| [p, z] \bigr|_{C([0, T]; H)^2}
            \\[1ex]
            \hspace{8.25ex} \leq (1 +T) (1 +C_{H}^{V^*} +C_{H}^{V_0^*}) \cdot C_0^* e^{\frac{3 C_0^* T}{2}} \bigl| \bigl[ [p_0, z_0], [h, k] \bigr] \bigr|_{[H]^2 \times \mathscr{Y}},
            \\[2ex]
            \bigl| [\partial_t p, \partial_t z] \bigr|_{\mathscr{Y}} \leq M_1 \bigl( \bigl| [h, k] \bigr|_{\mathscr{Y}} +\bigl| [p, z] \bigr|_{\mathscr{V} \times \mathscr{V}_0} \bigr) 
            \\[1ex]
            \hspace{12.225ex} \leq M_1 \bigl( 1 +C_0^* e^{\frac{3 C_0^* T}{2}} \bigr) \bigl| \bigl[ [p_0, z_0], [h, k] \bigr] \bigr|_{[H]^2 \times \mathscr{Y}},
        \end{cases}
    \end{equation}
    with use of the constants $ C_H^{V^*} > 0 $ and $ C_{H}^{V_0^*} > 0 $ of respective embeddings $ H \subset V^* $ and $ H \subset V_0^* $, and the constant $ M_1 $ given in \eqref{defOf_M1}. The above \eqref{concl1-10} and \eqref{concl1-11} imply that $ \mathcal{P} = \mathcal{P}(a, b, \mu, \lambda, \omega, A) \in \mathscr{L}\bigl( [H]^2 \times \mathscr{Y}; (\mathscr{Z}, \|\cdot\|) \bigr) $, and 
        \begin{align}\label{estFinal11}
            \bigl| \mathcal{P}(a, b, \mu, & \lambda, \omega, A) \bigr|_{\mbox{\scriptsize$ \mathscr{L}\bigl( [H]^2 \times \mathscr{Y};(\mathscr{Z}, \|\cdot\|) \bigr) $}}
            \nonumber
            \\
            \leq &  4 M_1(1 +T) (1 +C_{H}^{V^*} +C_{H}^{V_0^*}) \bigl( 1 +C_0^* e^{\frac{3}{2} C_0^* T} \bigr) =: M_1^*.
        \end{align}
        \eqref{estFinal10} and \eqref{estFinal11} are sufficient to verify \eqref{estFinal100}.
    \medskip
    
    Next, we show the item (\hyperlink{C1}{C1}). Since (\hyperlink{P}{P}) is a linear system, the decomposition \eqref{concl1-01} is obvious. Meanwhile, the linearity of (\hyperlink{P}{P}) and the identification $ [p_0, z_0] \in V \times V_0 \subset \mathscr{Z} $ allow us to verify \eqref{concl1-02} as follows:
    \begin{align*}
        \mathcal{P}\bigl( [H]^2 \times \mathscr{Y} \bigr) = & \bigcup_{[p_0, z_0] \in V \times V_0} \left( \mathcal{P} \bigl( \mathcal{Q}[p_0, z_0] \bigr) +\mathcal{P}\bigl(  \{ [0, 0] \} \times \bigl[ -\mathcal{T}[p_0, p_0, z_0, z_0] +\mathscr{Y} \bigr] \bigr) \right)
        \\
        = & \bigcup_{[p_0, z_0] \in V \times V_0} \bigl( [p_0, z_0] +\mathcal{P} \bigl( \{ [0, 0] \} \times \mathscr{Y} \bigr) \bigr) = [V \times V_0] +\mathcal{P} \bigl( \{ [0, 0] \} \times \mathscr{Y} \bigr).
    \end{align*}

    Thus, the proof is complete.
 \end{proof}

\begin{corollary}\label{Cor.concl2}
    Let us assume $ [a, b, \mu, \lambda, \omega, A] \in \mathscr{S}_0 $, $ \{ [a^n, b^n, \mu^n, \lambda^n, \omega^n, A^n] \}_{n = 1}^\infty \subset \mathscr{S}_0 $,
    \begin{subequations}\label{concl2-20}
        \begin{align}
            [a^n, & \, \partial_t a^n, \nabla a^n, b^n, \lambda^n, \omega^n, A^n] \to [a, \partial_t a, \nabla a, b, \lambda, \omega, A] \mbox{ weakly-$*$ in}
            \nonumber
            \\
            & L^\infty(Q) \times L^\infty(Q) \times L^\infty(Q)^N \times L^\infty(Q) \times L^\infty(Q) \times L^\infty(Q)^N \times L^\infty(Q)^{N \times N},
            \label{concl2-20_a}
            \\
            & \mbox{and in the pointwise sense a.e. in $ Q $,}
            \nonumber
        \end{align}
        and
        \begin{equation}\label{concl2-20_b}
            \begin{cases}
                \mu^n \to \mu \mbox{ weakly-$*$ in $ L^\infty(0, T; H) $,}
                \\
                \mu^n(t) \to \mu(t) \mbox{ in $ H $, for a.e. $ t \in (0, T) $,}
            \end{cases} \mbox{as $ n \to \infty $.}
        \end{equation}
    \end{subequations}
    Also, let us assume $ [p_0, z_0] \in [H]^2 $, $ \{ [p_0^n, z_0^n] \}_{n = 1}^\infty \subset [H]^2 $, $ [h, k] \in \mathscr{Y} $, $ \{ [h^n, k^n] \}_{n = 1}^\infty \subset \mathscr{Y} $, $ [p, z] \in \mathscr{Z} $, and $ \{[p^n, z^n]\}_{n = 1}^\infty \subset \mathscr{Z} $. In addition, we suppose the following relationship: 
    \begin{equation*}
        \begin{cases}
            [p, z] = \mathcal{P}(a, b, \mu, \lambda, \omega, A) \bigl[ [p_0, z_0], [h, k] \bigr] \mbox{ in } \mathscr{Z},
            \\
            [p^n, z^n] = \mathcal{P}(a^n, b^n, \mu^n, \lambda^n, \omega^n, A^n) \bigl[ [p_0^n, z_0^n], [h^n, k^n] \bigr] \mbox{ in } \mathscr{Z}, \mbox{ for $ n \in \bbN $,}
        \end{cases}
    \end{equation*}
    or equivalently,
    \begin{equation*}
        \begin{cases}
            \bigl[ [p_0, z_0], [h, k] \bigr] = \mathcal{Q}(a, b, \mu, \lambda, \omega, A) [p, z] \mbox{ in } [H]^2 \times \mathscr{Y},
            \\
            \bigl[ [p_0^n, z_0^n], [h^n, k^n] \bigr] = \mathcal{Q}(a^n, b^n, \mu^n, \lambda^n, \omega^n, A^n)[p^n, z^n] \mbox{ in } [H]^2 \times \mathscr{Y}, \mbox{ for $ n \in \bbN $.}
        \end{cases}
    \end{equation*}
    Then, the following two items hold. 
    \begin{description}
        \item[\textmd{\hypertarget{C2}{\textit{(C2)}}}]The convergence:
                \begin{subequations}\label{concl2-01}
                    \begin{equation}\label{concl2-01_a}
                        \bigl[ [p_0^n, z_0^n], [h^n, k^n] \bigr] \to \bigl[ [p_0, z_0], [h, k] \bigr] \mbox{ in $ [H]^2 \times \mathscr{Y} $, as $ n \to \infty $,}
                    \end{equation}
                implies the convergence:
                    \begin{equation}\label{concl2-01_b}
                        [p^n, z^n] \to [p, z] \mbox{ in $ C([0, T]; H)^2 $, and in $ \mathscr{V} \times \mathscr{V}_0 $, as $ n \to \infty $.}
                    \end{equation}
                    \vspace{-2ex}
                \end{subequations}
            \item[\textmd{\hypertarget{C3}{\textit{(C3)}}}]The following two convergences:
                \begin{subequations}\label{concl2-02}
                    \begin{equation}\label{concl2-02_a}
                        \bigl[ [p_0^n, z_0^n], [h^n, k^n] \bigr] \to \bigl[ [p_0, z_0], [h, k] \bigr] \mbox{ weakly in $ [H]^2 \times \mathscr{Y} $, as $ n \to \infty $,}
                    \end{equation}
                    and
                    \begin{equation}\label{concl2-02_b}
                        [p^n, z^n] \to [p, z] \mbox{ weakly in the Hilbert space $ \mathscr{Z} $, as $ n \to \infty $,}
                    \end{equation}
                \end{subequations}
            are equivalent to each other.
    \end{description}
\end{corollary}
\begin{proof}
    We begin by referring to the Rellich--Kondrachov theorem (cf. \cite[Theorem 6.3]{MR2424078}), and confirm that the assumption \eqref{concl2-20_a} implies the uniform convergence of $ \{a^n\}_{n = 1}^\infty $, i.e.:
        \begin{equation}\label{concl2-21}
            a^n \to a \mbox{ in $ C(\overline{Q}) $ as $ n \to \infty $.}
        \end{equation}
        So, we can find a large number $ N_a^* \in \bbN $ such that:
    \begin{equation}\label{Na*}
        0 < \frac{\delta_*(a)}{2} \leq \inf_{n \geq N_a^*} \delta_*(a^n) \leq \sup_{n \geq N_a^*} |a^n|_{C(\overline{Q})} \leq 2 |a|_{C(\overline{Q})}. 
    \end{equation}

    Now, we show the item (\hyperlink{C2}{C2}). Let us assume the convergence \eqref{concl2-01_a}, and let us apply Main Theorem \hyperlink{mainThrm01}{1} (\hyperlink{I-B}{I-B}) under the replacements of:
    \begin{equation}\label{replace01}
        \begin{cases}
            [a^1, b^1, \mu^1, \lambda^1, \omega^1, A^1] \mbox{ by } [a^n, b^n, \mu^n, \lambda^n, \omega^n, A^n], ~ [p_0^1, z_0^1] \mbox{ by }  [p_0^n, z_0^n], 
            \\
            [h^1, k^1] \mbox{ by } [h^n, k^n], \mbox{ and } [p^1, z^1] \mbox{ by } [p^n, z^n], \mbox{ for $ n \geq N_a^* $,}
        \end{cases}
    \end{equation}
and
    \begin{equation*}
        \begin{cases}
            [a^2, b^2, \mu^2, \lambda^2, \omega^2, A^2] \mbox{ by } [a, b, \mu, \lambda, \omega, A], ~ [p_0^2, z_0^2] \mbox{ by }  [p_0, z_0], 
            \\
            [h^2, k^2] \mbox{ by } [h, k], \mbox{ and } [p^2, z^2] \mbox{ by } [p, z].
        \end{cases}
    \end{equation*}
    Then, we have:
    \begin{align}\label{concl2-10}
        \bigl( |p^n - & p|_{C([0, T]; H)}^2 +|\sqrt{a^n}(z^n -z)|_{C([0, T]; H)}^2 \bigr) +\bigl( |p^n -p|_{\mathscr{V}}^2 +\nu |z^n -z|_{\mathscr{V}_0}^2 \bigr) 
        \nonumber
        \\
        \leq & 2e^{3 C_0^{**} T} \bigl( |p_0^n -p_0|_H^2 +|\sqrt{a^n(0)} (z_0^n -z)|_H^2 \bigr)
        \nonumber
        \\
        & + 4 C_0^{**} e^{3 C_0^{**} T} \bigl( |[h^n -h, k^n -k]|_\mathscr{Y}^2 +|R_0^n|_{L^1(0, T)} \bigr), \mbox{ for all $ n \geq N_a^* $,} 
    \end{align}
    where 
    \begin{align}\label{concl2-10-01}
        C_0^{**} := & \frac{18(1 +\nu)}{\min \{ 1, \nu, \delta_*(a) \}} \bigl( 1 +(C_V^{L^4})^2 +(C_V^{L^4})^4 +(C_{V_0}^{L^4})^2 \bigr) \cdot
        \nonumber
        \\
        & \cdot \sup_{n \geq N_a^*} \, \bigl\{ 1 +|a^n|_{W^{1, \infty}(Q)} +|b^n|_{L^\infty(Q)} +|\lambda^n|_{L^\infty(Q)} +|\omega^n|_{L^\infty(Q)^N} \bigr\},
    \end{align}
    and
    \begin{align}\label{concl2-11}
        R_0^n(t) := & |\partial_t z(t)|_{V_0^*}^2 \bigl( |a^n -a|_{C(\overline{Q})}^2 +|\nabla (a^n -a)(t)|_{L^4(\Omega)^N}^2 \bigr)
        \nonumber
        \\
        & +|z(t)|_{V_0}^2 |(b^n -b)(t)|_{L^4(\Omega)}^2 +|p(t)(\lambda^n -\lambda)(t)|_H^2
        \nonumber
        \\
        & +|p(t)|_V^2 \bigl( |(\mu^n -\mu)(t)|_H^2 +|(\omega^n -\omega)(t)|_{L^4(\Omega)^N}^2 \bigr)
        \nonumber
        \\
        & +|\nabla z(t) \cdot (\omega^n -\omega)(t)|_H^2 +|(A^n -A)(t) \nabla z(t)|_{[H]^N}^2, 
        \\
        & \qquad \mbox{for a.e. $ t \in (0, T) $, and all $ n \geq N_a^* $.}
        \nonumber
    \end{align}
    Notice that the assumption \eqref{concl2-20}, and the estimate \eqref{Na*} let the above $ C_0^{**} $ be a finite constant, and let it be an upper bound for the sequence of constants given in \eqref{defOf_C*}, under \eqref{Na*} and \eqref{replace01}. On this basis, we can verify the convergence \eqref{concl2-01_b}, by letting $ n \to \infty $ for \eqref{concl2-10} and \eqref{concl2-11}, with use of \eqref{concl2-20}, \eqref{concl2-21}, \eqref{Na*}, and the dominated convergence theorem (cf. \cite[Theorem 10]{MR0492147}). 
    \medskip

    Next, we show the item (\hyperlink{C3}{C3}). If the convergence \eqref{concl2-02_a} holds, then thanks to \eqref{concl2-20}, \eqref{Na*}--\eqref{concl2-11}, we can say that 
        \begin{equation}\label{bdd_a}
            \{ [p^n, z^n] \}_{n = N_a^*}^\infty \mbox{ is bounded in $ C([0, T]; H)^2 $, and in $ \mathscr{V} \times \mathscr{V}_0 $.} 
        \end{equation}
        Also, the assumption \eqref{concl2-20}, and the estimate \eqref{Na*} lead to the boundedness of a sequence \linebreak $ \{ M_1(a^n, b^n, \mu^n, \lambda^n, \omega^n, A^n) \}_{n = N_a^*}^\infty $, which consists of the constants $ M_1 = M_1(a, b, \mu, \lambda, \omega, A) $ given in \eqref{defOf_M1}, in cases when:
        \begin{equation}\label{cases}
            [a, b, \mu, \lambda, \omega, A] = [a^n, b^n, \mu^n, \lambda^n, \omega^n, A^n], \mbox{ for $ n \geq N_a^* $.}
        \end{equation}
        Therefore, from \eqref{concl2-02_a}, \eqref{bdd_a}, and the estimate \eqref{estFinal01} as in Remark \ref{Rem.4concl1}, it is inferred that:
        \begin{equation}\label{bdd_b}
            \{ [\partial_t p^n, \partial_t z^n] \}_{n = N_a^*}^\infty \mbox{ is bounded in $ \mathscr{Y} $.}
        \end{equation}
        On account of \eqref{bdd_a} and \eqref{bdd_b}, we can ensure that:
        
        \begin{equation}\label{concl2-13}
            \begin{array}{c}
                [p^n, z^n] \to [\bar{p}, \bar{z}] \mbox{ in $ [\scrH]^2 $, in $ C([0, T]; V^*) \times C([0, T]; V_0^*) $,  weakly in $ \mathscr{Z} $,}
                \\[1ex]
                \mbox{and weakly-$*$ in $ L^\infty(0, T; H)^2 $, as $ n \to \infty $,}
            \end{array}
        \end{equation}
        and in particular, 
        \begin{align}
            [\bar{p}(0), & \bar{z}(0)] = \lim_{n \to \infty} [p^n(0), z^n(0)] = \lim_{n \to \infty} [p_0^n, z_0^n] = [p_0, z_0]
            \nonumber
            \\
            & \mbox{in the topology of $ V^* \times V_0^* $, and in the weak topology of $ [H]^2 $,}
            \label{concl2-15}
        \end{align}

        \noindent
        for some $ [\bar{p}, \bar{z}] \in \mathscr{Z} $, and some subsequence, by applying the weakly-$*$ compact embedding $ \mathscr{Z} \subset L^\infty(0, T; H)^2 $ as in Remark \ref{Rem3.concl1}, and the theories of compactness of Aubin's type \cite[Corollary 4]{MR0916688}, Arzer\`{a}--Ascoli \cite[Theorem 1.3.1]{MR1375237}, and Alaoglu--Bourbaki--Kakutani \cite[Theorem 1.2.5]{MR1375237}. Furthermore, under the replacements of:
        \begin{equation*}
            \{ [\tilde{p}^n, \tilde{z}^n] \}_{n = 1}^\infty \mbox{ by } \{ [p^n, z^n] \}_{n = N_a^*}^\infty, ~  [\tilde{p}, \tilde{z}] \mbox{ by } [\bar{p}, \bar{z}], \mbox{ and } [\tilde{h}, \tilde{k}] \mbox{ by } [h, k], 
        \end{equation*}
        we can derive the assumption \eqref{axLem02-02} of Key-Lemma \ref{axLem02} from \eqref{concl2-20}, \eqref{concl2-02_a}, \eqref{concl2-13}, and the dominated convergence theorem (cf. \cite[Theorem 10]{MR0492147}). So, applying Key-Lemma \ref{axLem02}, one can observe that: 
        \begin{equation}\label{concl2-14}
            \mathcal{T}(a, b, \mu, \lambda, \omega, A)[\bar{p}, \bar{p}, \bar{z}, \bar{z}] = [h, k] \mbox{ in $ \mathscr{Y} $.}
        \end{equation}
        Taking into account \eqref{concl2-15}, \eqref{concl2-14}, and the uniqueness result of Main Theorem \hyperlink{mainThrm01}{1} \linebreak (\hyperlink{I-A}{I-A}), we deduce that $ [\bar{p}, \bar{z}] = [p ,z] $, and the convergence \eqref{concl2-13} holds without necessity of subsequence. This implies the convergence \eqref{concl2-02_b}.
        \medskip

        Conversely, if the convergence \eqref{concl2-02_b} holds, then from the weakly-$*$ compact embedding $ \mathscr{Z} \subset L^\infty(0, T; H)^2 $ as in Remark \ref{Rem3.concl1}, and the compactness theories of Aubin's type \cite[Corollary 4]{MR0916688}, and Arzer\`{a}--Ascoli \cite[Theorem 1.3.1]{MR1375237}, we can see that:
            \begin{align}\label{Oh01}
                [p^n, z^n] \to [p, z] & \mbox{ in $ [\scrH]^2 $, in $ C([0, T]; V^*) \times C([0, T]; V_0^*) $,}
                \nonumber
                \\
                \mbox{and } & \mbox{weakly-$ * $ in $ L^\infty(0, T; H)^2 $, as $ n \to \infty $,}
            \end{align}
            and in particular,
            \begin{align}\label{concl2-40}
                [p_0^n, z_0^n] := & [p^n(0), z^n(0)] \to [p(0), z(0)] =: [p_0, z_0] 
                \nonumber
                \\
                & \mbox{in $ V^* \times V_0^* $, and weakly in $ [H]^2 $, as $ n \to \infty $,}
            \end{align}
            by taking a subsequence. Also, the assumption \eqref{concl2-20} and the condition \eqref{Na*} lead to the boundedness of a sequence $ \{ M_0(a^n, b^n, \mu^n, \lambda^n, \omega^n, A^n) \}_{n = N_a^*}^\infty $ which consists of the constants $ M_0 = M_0(a, b, \mu, \lambda, \omega, A) $ given in \eqref{defOf_M0}, in the cases as in \eqref{cases}. So, from the estimate \eqref{estFinal00} as in Remark \ref{Rem.4concl1}, it is observed that the sequence $ \{ [h^n, k^n] \}_{n = N_a^*}^\infty := \{ \mathcal{T}(a^n, b^n, \mu^n, \lambda^n, \omega^n, A^n)[p^n, p^n, z^n, z^n] \}_{n = N_a^*}^\infty $ is bounded in  $ \mathscr{Y} $, and by taking more one subsequence if necessary, 
            \begin{align}\label{concl2-41}
                [h^n, k^n] \to & [\bar{h}, \bar{k}] \mbox{ weakly in $ \mathscr{Y} $ as $ n \to \infty $, for some $ [\bar{h}, \bar{k}] \in \mathscr{Y} $.}
            \end{align}
            Additionally, \eqref{concl2-20}, \eqref{concl2-02_b}, \eqref{Oh01}, and \eqref{concl2-41} allows us to apply Key-Lemma \ref{axLem02} under the replacements of:
            \begin{equation*}
                [\tilde{p}, \tilde{z}] \mbox{ by } [p, z], ~ \{ [\tilde{p}^n, \tilde{z}^n] \}_{n = 1}^\infty \mbox{ by } \{ [p^n, z^n] \}_{n = N_a^*}^\infty, \mbox{ and } [\tilde{h}, \tilde{k}] \mbox{ by } [\bar{h}, \bar{k}],
            \end{equation*}
            and deduce that:
            \begin{equation}\label{concl2-42}
                [\bar{h}, \bar{k}] = [h, k] := \mathcal{T}(a, b, \mu, \lambda, \omega, A)[p, p, z, z] \mbox{ in $ \mathscr{Y} $.}
            \end{equation}
            By virtue of \eqref{concl2-40}--\eqref{concl2-42}, we can conclude the convergence \eqref{concl2-02_a}, because \eqref{concl2-40}, \eqref{concl2-42}, and the uniqueness result in Main Theorem \hyperlink{mainThrm01}{1} (\hyperlink{I-A}{I-A}) guarantee the unnecessity of subsequence in \eqref{concl2-40} and \eqref{concl2-41}. 
        \medskip

        Thus, the proof of this Corollary is complete. 
\end{proof}

\begin{remark}\label{Rem.add01}
    By Corollary \ref{Cor.concl2}, we can say that \eqref{concl2-20} is a sufficient assumption in the application to the optimal control problem (\hyperlink{OCP}{OCP}) mentioned in Introduction. However, this assumption  is not sufficient to make the convergences \eqref{concl2-01_a} and \eqref{concl2-01_b} to be equivalent to each other. 
\end{remark}

\begin{crem}\label{Rem.concl}
    The first Corollary \ref{Cor.concl1} suggests us that the linear operator \linebreak $ \mathcal{P} \bigl[ [0, 0], \cdot \, \bigr] = \mathcal{P}(a, b, \mu, \lambda, \omega, A) \bigl[ [0, 0], \cdot \, \bigr] : \mathscr{Y} \longrightarrow \mathscr{Z} $ takes over one of  essential roles in the mathematical analysis of our system (\hyperlink{P}{P}), and especially, all cases of $ [p_0, z_0] \in V \times V_0 $ are reduced to the case of $ [p_0, z_0] = [0, 0] $. Additionally, we note that the reduced initial condition ``$ [p_0, z_0] = [0, 0] $'' is a natural setting in the application to the problems $ \hyperlink{sharp1}{\sharp\,1}) $ and $ \hyperlink{sharp2}{\sharp\,2}) $, under the corresponding relations \eqref{relat_P-S01} and \eqref{relat_P-S02}.
\medskip

Finally, as a result of Main Theorem \hyperlink{mainThrm01}{1} (\hyperlink{I-B}{I-B}) and Key-Lemma \ref{axLem01}, we can conclude that our numerical scheme for the system (\hyperlink{P}{P}) is designed as the minimization algorithm for a single governing energy $ \mathcal{E} $ given  in \eqref{enrgyE}.  
\end{crem}

\label{page:e}

\begin{thebibliography}{10}

\bibitem{MR2424078}
Adams, R.~A.; Fournier, J.~J.~F.
\newblock {\em Sobolev spaces}, Vol. 140 of {\em Pure and Applied Mathematics
  (Amsterdam)}.
\newblock Elsevier/Academic Press, Amsterdam, second edition, 2003.

\bibitem{MR2192832}
Attouch, H.; Buttazzo, G.; Michaille, G.
\newblock {\em Variational Analysis in {S}obolev and {BV} spaces}, Vol.~6 of
  {\em MPS/SIAM Series on Optimization}.
\newblock Society for Industrial and Applied Mathematics (SIAM), Philadelphia,
  PA; Mathematical Programming Society (MPS), Philadelphia, PA, 2006.
\newblock Applications to PDEs and optimization.

\bibitem{MR0348562}
Br\'ezis, H.
\newblock {\em Op\'erateurs Maximaux Monotones et Semi-groupes de Contractions
  dans les Espaces de {H}ilbert}.
\newblock North-Holland Publishing Co., Amsterdam-London; American Elsevier
  Publishing Co., Inc., New York, 1973.
\newblock North-Holland Mathematics Studies, No. 5. Notas de Matem\'atica (50).

\bibitem{emmrich1999discrete}
Emmrich, E.
\newblock Discrete versions of gronwall's lemma and their application to the
  numerical analysis of parabolic problems.
\newblock Technical Report 637, Institute of Mathematics, Technische
  Universit\"{a}t Berlin,
  ``\href{http://www3.math.tu-berlin.de/preprints/files/Preprint-637-1999.pdf}{http://www3.math.tu-berlin.de/preprints/files/Preprint-637-1999.pdf}'',
  1999.

\bibitem{MR2469586}
Ito, A.; Kenmochi, N.; Yamazaki, N.
\newblock A phase-field model of grain boundary motion.
\newblock {\em Appl. Math.}, {\bfseries 53}(5): 433--454, 2008.

\bibitem{MR2548486}
Ito, A.; Kenmochi, N.; Yamazaki, N.
\newblock Weak solutions of grain boundary motion model with singularity.
\newblock {\em Rend. Mat. Appl. (7)}, {\bfseries 29}(1): 51--63, 2009.

\bibitem{MR2836555}
Ito, A.; Kenmochi, N.; Yamazaki, N.
\newblock Global solvability of a model for grain boundary motion with
  constraint.
\newblock {\em Discrete Contin. Dyn. Syst. Ser. S}, {\bfseries 5}(1): 127--146,
  2012.

\bibitem{MR2668289}
Kenmochi, N.; Yamazaki, N.
\newblock Large-time behavior of solutions to a phase-field model of grain
  boundary motion with constraint.
\newblock In {\em Current advances in nonlinear analysis and related topics},
  Vol.~32 of {\em GAKUTO Internat. Ser. Math. Sci. Appl.}, pp.  389--403.
  Gakk\=otosho, Tokyo, 2010.

\bibitem{MR1752970}
Kobayashi, R.; Warren, J.~A.; Carter, W.~C.
\newblock A continuum model of grain boundaries.
\newblock {\em Phys. D}, {\bfseries 140}(1-2): 141--150, 2000.

\bibitem{MR1794359}
Kobayashi, R.; Warren, J.~A.; Carter, W.~C.
\newblock Grain boundary model and singular diffusivity.
\newblock In {\em Free boundary problems: theory and applications, {II}
  ({C}hiba, 1999)}, Vol.~14 of {\em GAKUTO Internat. Ser. Math. Sci. Appl.},
  pp.  283--294. Gakk\=otosho, Tokyo, 2000.

\bibitem{MR0259693}
Lions, J.-L.
\newblock {\em Quelques m\'ethodes de r\'esolution des probl\`emes aux limites
  non lin\'eaires}.
\newblock Dunod; Gauthier-Villars, Paris, 1969.

\bibitem{MR0492147}
Mikusi\'nski, J.
\newblock {\em The {B}ochner integral}.
\newblock Birkh\"auser Verlag, Basel-Stuttgart, 1978.
\newblock Lehrb\"ucher und Monographien aus dem Gebiete der exakten
  Wissenschaften, Mathematische Reihe, Band 55.

\bibitem{MR3670006}
Moll, S.; Shirakawa, K.; Watanabe, H.
\newblock Energy dissipative solutions to the {K}obayashi--{W}arren--{C}arter
  system.
\newblock {\em Nonlinearity}, {\bfseries 30}(7): 2752--2784, 2017.

\bibitem{Nakayashiki18}
Nakayashiki, R.
\newblock Quasilinear type {K}obayashi--{W}arren--{C}arter systems including
  dynamic boundary conditions.
\newblock {\em Adv. Math. Sci. Appl. (to appear)}.

\bibitem{MR3038131}
Shirakawa, K.; Watanabe, H.; Yamazaki, N.
\newblock Solvability of one-dimensional phase field systems associated with
  grain boundary motion.
\newblock {\em Math. Ann.}, {\bfseries 356}(1): 301--330, 2013.

\bibitem{MR3082861}
Shirakawa, K.; Watanabe, H.
\newblock Energy-dissipative solution to a one-dimensional phase field model of
  grain boundary motion.
\newblock {\em Discrete Contin. Dyn. Syst. Ser. S}, {\bfseries 7}(1): 139--159,
  2014.

\bibitem{MR3362773}
Shirakawa, K.; Watanabe, H.; Yamazaki, N.
\newblock Phase-field systems for grain boundary motions under isothermal
  solidifications.
\newblock {\em Adv. Math. Sci. Appl.}, {\bfseries 24}(2): 353--400, 2014.

\bibitem{MR0916688}
Simon, J.
\newblock Compact sets in the space {$L^p(0,T;B)$}.
\newblock {\em Ann. Mat. Pura Appl. (4)}, {\bfseries 146}: 65--96, 1987.

\bibitem{MR1375237}
Vrabie, I.~I.
\newblock {\em Compactness methods for nonlinear evolutions}, Vol.~75 of {\em
  Pitman Monographs and Surveys in Pure and Applied Mathematics}.
\newblock Longman Scientific \& Technical, Harlow; copublished in the United
  States with John Wiley \& Sons, Inc., New York, second edition, 1995.
\newblock With a foreword by A. Pazy.

\bibitem{MR3203495}
Watanabe, H.; Shirakawa, K.
\newblock Qualitative properties of a one-dimensional phase-field system
  associated with grain boundary.
\newblock In {\em Nonlinear analysis in interdisciplinary
  sciences---modellings, theory and simulations}, Vol.~36 of {\em GAKUTO
  Internat. Ser. Math. Sci. Appl.}, pp.  301--328. Gakk\=otosho, Tokyo, 2013.

\end{thebibliography}
\end{document}